\documentclass[11pt]{article}
\usepackage{amsfonts, amssymb, amsmath, amsthm, xcolor, graphicx, pstricks, pstricks-add}
\begin{document}

\theoremstyle{plain}
\newtheorem{thm}{Theorem}[section]
\newtheorem{cor}[thm]{Corollary}
\newtheorem{con}[thm]{Conjecture}
\newtheorem{cla}[thm]{Claim}
\newtheorem{lm}[thm]{Lemma}
\newtheorem{prop}[thm]{Proposition}

\theoremstyle{definition}
\newtheorem{dfn}[thm]{Definition}
\newtheorem{alg}[thm]{Algorithm}
\newtheorem{prob}[thm]{Problem}
\newtheorem{example}[thm]{Example}
\newtheorem{rem}[thm]{Remark}

\renewcommand{\baselinestretch}{1.1}

\title{\bf Gallai-Edmonds Structure Theorem for Weighted Matching Polynomial}
\author{
Cheng Yeaw Ku
\thanks{ Department of Mathematics, National University of Singapore, Singapore 117543. E-mail: matkcy@nus.edu.sg} \and Kok Bin Wong \thanks{
Institute of Mathematical Sciences, University of Malaya, 50603 Kuala Lumpur, Malaysia. E-mail:
kbwong@um.edu.my.} } \maketitle

\begin{abstract}\noindent
In this paper, we prove the Gallai-Edmonds structure theorem for the most general matching polynomial. Our result implies the Parter-Wiener theorem and its recent generalization about the existence of principal submatrices of a Hermitian matrix  whose graph is a tree.
\end{abstract}

\bigskip\noindent
{\sc keywords:} matching polynomial, characteristic polynomial, Gallai-Edmond decomposition, Hermitian matrices, Parter-Wiener theorem

\section{Introduction}

Recently, Chen and Ku \cite{KC} proved an analogue of the celebrated Gallai-Edmonds structure theorem for general roots of the matching polynomial. Their result implies that every connected vertex transitive graph has simple matching polynomial roots. Subsequently, following a line of investigation pursued by Lov\'{a}sz and Plummer \cite{Lo}, Ku and Wong wrote a series of papers \cite{KW1, KW2, KW3, KW4, KW5, KW6} to develop a matching theory for general roots of the matching polynomial. In this paper, we shall prove the Gallai-Edmonds structure theorem for the most general matching polynomial. Surprisingly, our result implies the Parter-Wiener theorem and its recent generalization by Johnson, Duarte and Saiago \cite{Johnson2} about the existence of principal submatrices of a Hermitian matrix  whose graph is a tree.

All graphs in this paper are simple and finite. The vertex set and the edge set of a graph $G$ will be denoted by $V(G)$ and $E(G)$ respectively. Recall that an $r$-\emph {matching} in a graph $G$ is a set of $r$ edges, no two of which have a vertex in common. The number of $r$-matchings in $  G$ will be denoted by $p( G,r)$. We set $p(G,0)=1$ and define the \emph {matching polynomial} of $G$ by
\begin {equation}
\mu ( G,x)=\sum_{r=0}^{\lfloor n/2\rfloor} (-1)^rp(G,r)x^{n-2r},\notag
\end {equation}
where $n=\vert V(G)\vert$.

In this paper we shall consider weighted versions of the matching polynomial. From now on, we assign a non-zero complex number $w(e)$ to every edge $e$ of our graph $G$ (we shall give a reason why we do not want $w$ to take zero value later). We can view $w$ as a function on $E(G)$ and call $w$ the \emph{edge weight function}. We also denote an edge by $e_{uv}$ to emphasize that the edge has endpoints $u$ and $v$. For each complex number $y=a+bi\in\mathbb C$, we denote its conjugate by $\overline y=a-bi$ and its magnitude by $\vert y\vert=\sqrt{a^2+b^2}$. Also for any set $S$, we denote the number of elements in $S$ by $\vert S\vert$. Although the notations for the magnitude of a complex number and the number of element in a set look similar, they will not cause any confusion.

For each $ A \subseteq E(G)$, we define $w(A)=\prod_{e \in A} w(e)$. We set $w(\varnothing)=1$. Let $\mathcal{M}(G)$ denote the set of all matchings of $G$ including the empty set $\varnothing$. The \emph{edge weighted matching polynomial} of $G$ is defined by

\[ \mu_{w}(G,x) = \sum_{M \in \mathcal{M}(G)} (-1)^{|M|} \vert w(M)\vert^2 x^{n-2|M|}. \]

We denote the set of all $r$-matchings in $G$ by $M_r(G)$ and set $M_0(G)=\{\varnothing\}$. The following lemma is obvious from the definition.

\begin{lm}\label{different_form_generalized_matching}
\[ \mu_{w}(G,x) =\sum_{r=0}^{\lfloor n/2\rfloor}(-1)^r\left (\sum_{M \in M_r(G)} \vert w(M)\vert^2\right) x^{n-2r}, \]
where $n=\vert V(G)\vert$.\qed
\end{lm}

Using Lemma \ref{different_form_generalized_matching}, it is not hard to deduce the followings.

\begin{lm}\label{zer0_root_generalized_matching} For any edge weight function $w$, zero is a root of $\mu_{w}(G,x)$ if and only if $G$ does not have a perfect matching. \qed
\end{lm}

\begin{lm}\label{edge_weighted_to_matching} Suppose $w(e)=1$ for all $e \in E(G)$. Then $\mu_w(G,x)=\mu(G,x)$. \qed
\end{lm}

By Lemma \ref{zer0_root_generalized_matching}, the fact that zero is a root of $\mu_{w}(G,x)$ depends only on the structure of the graph $G$  and does not depend on the edge weight function. By Lemma \ref{edge_weighted_to_matching}, if  the edge weight function takes only the value 1, then the edge weighted matching polynomial is the usual matching polynomial.

Let $u\in V(G)$. The graph obtained from $G$ by deleting the vertex $u$ and all edges that contain $u$ will be denoted by $G\setminus u$. The weight function on $G\setminus u$ is induced by the weight function $w$ on $G$. Inductively if $u_1,\dots, u_k\in V(G)$, $G\setminus u_1\dots u_k=(G\setminus u_1\dots u_{k-1})\setminus u_k$. For convenience if $H$ is a subgraph of $G$ then we shall denote $G\setminus V(H)$ by $G\setminus H$. If $e_1,\dots, e_m\in E(G)$ then  the graph obtained from $G$ by deleting all the edges $e_1,\dots ,e_m$ will  be denoted by $G-e_1\dots e_m$. The weight function on $G-e_1\dots e_m$ is induced by the weight function $w$ on $G$.

If we were to allow  $w$ to take zero value then $\mu_{w}(G,x) =\mu_{w}(G-e_1\dots e_m,x)$ where $w(e_1)=\cdots=w(e_k)=0$. So we may remove the edges with zero weight and the resulting graph has the same edge weighted matching polynomial. This is the reason why we do not allow $w$ to take zero value.

 The edge weighted matching polynomial $\mu_{w}(G,x)$ is a special case of the original multivariate matching polynomial introduced by Heilmann and Lieb \cite{HL}, who proved that all roots of $\mu_w(G,x)$ are real (\cite[Theorem 4.2]{HL}). As a consequence,  the roots of the usual matching polynomial $\mu(G,x)$ are real (Lemma \ref{edge_weighted_to_matching}). This fact was also proved by Godsil in his book  \cite[Corollary 1.2 on p. 97]{G0} via the classical recursive approach (see also \cite[Corollary 5.2]{GG}). Recently, by generalizing Foata's combinatorial proof of the Mehler formula for Hermite polynomials to matching polynomials, Lass \cite[Corollary on p. 439]{BL} proved that all the roots of $\mu_{w}(G,x)$ are real.

Now let us further generalize the edge weighted matching polynomial by assigning a real number $w_1(u)$ to every vertex $u$ of our graph $G$ (we allow $w_1$ to take zero value). We can view $w_1$ as a function on $V(G)$ and call $w_1$ the \emph{vertex weight function}. The pair $(w,w_1)$ will be called the \emph{weight function}. For each non-empty set $S\subseteq V(G)$, let $H_G(S)$ be the subgraph of $G$ induced by the vertices in $S$, that is $V(H_G(S))=S$ and $e_{uv}\in E(H_G(S))$ if and only if $e_{uv}\in E(G)$ and $u,v\in S$.

For each $S\subseteq V(G)$, we define $w_1(G\setminus S)=\prod_{u \in V(G\setminus S)} w_1(u)$. We set $w_1(\varnothing)=1$, $H_G(\varnothing)=\varnothing$ and $\mu_{w} (\varnothing,x)=1$. The \emph{weighted matching polynomial} of $G$ is defined by

\[ \eta_{(w,w_1)}(G,x) =\sum_{S\subseteq V(G)}(-1)^{\vert V(G\setminus S)\vert}w_1(G\setminus S)\mu_{w} (H_G(S),x). \]
It turns out that the weighted matching polynomial can be rewritten as

\[ \eta_{(w,w_1)}(G,x) = \sum_{M \in \mathcal{M}(G)} \left( \left(\prod_{e \in M} w(e) \right) \left( \prod_{u \in V(G) \setminus V(M)} (x-w_{1}(u))\right)\right) \]
which was proved by Averbouch and Makowsky \cite{AM} to be the most general nontrivial polynomial satisfying the matching polynomial recurrence relations.

\begin{example}\label{example_1} Let $G$ be the graph in Figure 1. Let $w(e_{u_1u_2})=2+i$, $w_1(u_1)=1$ and $w_1(u_2)=3$. Note that all the possible subsets of $V(G)$ are $S_1=\varnothing$, $S_2=\{u_1\}$, $S_3=\{u_2\}$ and $S_4=\{u_1,u_2\}$. Now $\mu_{w} (H_G(S_1),x)=1$, $\mu_{w} (H_G(S_2),x)=x$, $\mu_{w} (H_G(S_3),x)=x$ and $\mu_{w} (H_G(S_4),x)=x^2-\vert 2+i\vert^2=x^2-5$. Also $w_1(G\setminus S_1)=w_1(u_1)w_1(u_2)=3$, $w_1(G\setminus S_2)=w_1(u_2)=3$, $w_1(G\setminus S_3)=w_1(u_1)=1$ and $w_1(G\setminus S_4)=1$. Therefore $\eta_{(w,w_1)}(G,x)=(x^2-5)-(1)x-(3)x+3=x^2-4x-2$.

\begin{center}
\begin{pspicture}(0,0)(3,2)
\cnodeput(1, 1){1}{}
\cnodeput(2, 1){2}{}
\ncline{1}{2}
\rput(1.5,0){\textnormal{Figure 1}.}
\rput(0,1){$G=$}
\rput(1,1.3){$u_1$}
\rput(2,1.3){$u_2$}
\end{pspicture}
\end{center}\qed
\end{example}

\begin{example}\label{example_2} Let $G$ be the graph in Figure 2. Let $w(e_{v_1v_2})=1+2i$, $w(e_{v_2v_3})=2-7i$, $w(e_{v_1v_3})=-3+2i$, $w_1(v_1)=1$, $w_1(v_2)=2$ and $w_1(v_3)=3$. Note that all the possible subsets of $V(G)$ are $S_1=\varnothing$, $S_2=\{v_1\}$, $S_3=\{v_2\}$, $S_4=\{v_3\}$, $S_5=\{v_1,v_2\}$, $S_6=\{v_1,v_3\}$, $S_7=\{v_2,v_3\}$, $S_8=\{v_1,v_2, v_3\}$. Now $\mu_{w} (H_G(S_1),x)=1$, $\mu_{w} (H_G(S_2),x)=x$, $\mu_{w} (H_G(S_3),x)=x$, $\mu_{w} (H_G(S_4),x)=x$, $\mu_{w} (H_G(S_5),x)=x^2-5$, $\mu_{w} (H_G(S_6),x)=x^2-13$, $\mu_{w} (H_G(S_7),x)=x^2-53$ and $\mu_{w} (H_G(S_8),x)=x^3-(5+13+53)x=x^3-71x$. Also $w_1(G\setminus S_1)=w_1(v_1)w_1(v_2)w_1(v_3)=6$, $w_1(G\setminus S_2)=w_1(v_2)w_1(v_3)=6$, $w_1(G\setminus S_3)=w_1(v_1)w_1(v_3)=3$, $w_1(G\setminus S_4)=w_1(v_1)w_1(v_2)=2$, $w_1(G\setminus S_5)=w_1(v_3)=3$, $w_1(G\setminus S_6)=w_1(v_2)=2$, $w_1(G\setminus S_7)=w_1(v_1)=1$ and $w_1(G\setminus S_8)=1$. Therefore $\eta_{(w,w_1)}(G,x)=(x^3-71x)-(x^2-53)-2(x^2-13)-3(x^2-5)+2(x)+3(x)+6(x)-6=x^3-6x^2-60x+88$.

\begin{center}
\begin{pspicture}(0,0)(4,3)
\cnodeput(1, 1){1}{}
\cnodeput(1.5, 2){2}{}
\cnodeput(2, 1){3}{}
\ncline{1}{2}
\ncline{1}{3}
\ncline{2}{3}
\rput(1.5,0){\textnormal{Figure 2}.}
\rput(0,1){$G=$}
\rput(1,0.7){$v_1$}
\rput(1.5,2.3){$v_2$}
\rput(2,0.7){$v_3$}
\end{pspicture}
\end{center}\qed
\end{example}

For consistency, we set $\eta_{(w,w_1)}(\varnothing,x) =1$. The following three lemmas are obvious.

\begin{lm}\label{weighted_to_edge_weighted_matching} If $w_1(u)=0$ for all $u\in V(G)$ then
\[ \eta_{(w,w_1)}(G,x) =\mu_{w} (G,x). \] \qed
\end{lm}

\begin{lm}\label{weighted_matching_with_zero_vertex_weight} Let $u_1,\dots,u_m\in V(G)$ be such that $w_1(u_1)=\cdots=w_1(u_m)=0$. Then
\[ \eta_{(w,w_1)}(G,x) =\sum_{\substack {S\subseteq V(G),\\ \{u_1,\dots, u_m\}\subseteq S}}(-1)^{\vert V(G\setminus S)\vert}w_1(G\setminus S)\mu_{w} (H_G(S),x). \] \qed
\end{lm}

\begin{lm}\label{degree_weighted} The degree of $\eta_{(w,w_1)}(G,x)$ is equal to the degree of $\mu_{w}(G,x)$, which is $\vert V(G)\vert$. \qed
\end{lm}

Let $G_1$ and $G_2$ be graphs with weight function $(w,w_1)$ and $(w',w_1')$, respectively. The two graphs are said to be \emph{weight-isomorphic} if there is a bijection $f :V(G_1)\rightarrow V(G_2)$ such that
\begin{itemize}
\item [(a)] $e_{f(u)f(v)}\in E(G_2)$ if and only if $e_{uv}\in E(G_1)$,
\item [(b)] $w'(e_{f(u)f(v)})=w(e_{uv})$ for all $e_{uv}\in E(G_1)$,
\item [(c)] $w_1'(f(u))=w_1(u)$ for all $u\in V(G_1)$.
\end{itemize}
Note that if conditions (b) and (c) are removed then this is just the `usual' isomorphism.

\begin{example}\label{example_isomorphic} Let $G_1$ and $G_2$ be the graphs in Figure 3. The edge weight functions for both graphs take value 1 for all the edges, whereas the vertex weight functions are as stated. Note that they are not weight-isomorphic (even though they are isomorphic in the `usual' sense when the weights are removed).

\begin{center}
\begin{pspicture}(0,0)(6,3)
\cnodeput(1, 2){1}{}
\cnodeput(1.5, 1.5){2}{}
\cnodeput(1.5, 1){3}{}
\cnodeput(2, 2){4}{}
\cnodeput(4, 2){5}{}
\cnodeput(4.5, 1.5){6}{}
\cnodeput(4.5, 1){7}{}
\cnodeput(5, 2){8}{}
\ncline{1}{2}
\ncline{2}{3}
\ncline{2}{4}
\ncline{5}{6}
\ncline{6}{7}
\ncline{6}{8}
\rput(3,0){\textnormal{Figure 3}.}
\rput(0.5,1.5){$G_1=$}
\rput(3.5,1.5){$G_2=$}
\rput(0.7,2){2}
\rput(1.5,0.7){3}
\rput(1.8,1.5){3}
\rput(2.3,2){4}
\rput(3.7,2){2}
\rput(4.5,0.7){3}
\rput(4.8,1.5){9}
\rput(5.3,2){4}
\end{pspicture}
\end{center}\qed
\end{example}

The following lemma can be proved easily.

\begin{lm}\label{isomorphi_weighted} Let $G_1$ and $G_2$ be graphs with weight function $(w,w_1)$ and $(w',w_1')$, respectively. If $G_1$ is weight-isomorphic to $G_2$, then $\eta_{(w,w_1)}(G_1,x)=\eta_{(w',w_1')}(G_2,x)$. \qed
\end{lm}

Now by Lemma \ref{weighted_to_edge_weighted_matching}, the weighted matching polynomial $\eta_{(w,w_1)}(G,x)$ is a generalization of the edge weighted matching polynomial $\mu_w(G,x)$. So it is quite natural to ask whether the roots of $\eta_{(w,w_1)}(G,x)$ are real or not. In Section 3, we give an affirmative answer using Godsil's approach \cite{G0} (Corollary \ref{root_graph}). This generalizes the result of Lass \cite[Corollary on p. 439]{BL}.

Let $G$ be a graph with $V(G)=\{1,2,\dots, n\}$ and $B_{(w,w_1)}(G)=[b_{uv}]$ be the $n\times n$ matrix with
\begin{equation}
b_{uv}=\begin{cases}
w(e_{uv}), & \textnormal{if $e_{uv}\in E(G)$ and $u<v$;}\\
w_1(u), & \textnormal{if $u=v$;}\\
\overline {w(e_{vu})}, & \textnormal{if $e_{vu}\in E(G)$ and $u>v$;}\\
0, & \textnormal{otherwise.}
\end{cases}\notag
\end{equation}
We call $B_{(w,w_1)}(G)$ the \emph{weighted adjacency matrix} of $G$. Note that $B_{(w,w_1)}(G)$ is a \emph{Hermitian} matrix, that is $\overline {b_{uv}}=b_{vu}$ for all $u,v$. The \emph{weighted characteristic polynomial} of $G$ is defined by
\begin {equation}
\phi_{(w,w_1)}(G,x)=\det (xI-B_{(w,w_1)}(G)).\notag
\end {equation}

\begin{example}\label{example_3} Let $G$ and $(w,w_1)$ be as given in Example \ref{example_1}. Here we assume $V(G)=\{1,2\}$ where $u_1\equiv 1$ and $u_2\equiv 2$.
Then
\begin{equation}
B_{(w,w_1)}(G)=\begin{pmatrix}
1 & 2+i\\
2-i & 3
\end{pmatrix},\notag
\end{equation}
and $\phi_{(w,w_1)}(G,x)=x^2-4x-2$.\qed
\end{example}

\begin{example}\label{example_4} Let $G$ and $(w,w_1)$ be as given in Example \ref{example_2}.  Here we assume $V(G)=\{1,2,3\}$ where $v_1\equiv 1$, $v_2\equiv 2$ and $v_3\equiv 3$.
Then
\begin{equation}
B_{(w,w_1)}(G)=\begin{pmatrix}
1 & 1+2i & -3+2i\\
1-2i & 2 & 2-7i\\
-3-2i & 2+7i &3
\end{pmatrix},\notag
\end{equation}
and $\phi_{(w,w_1)}(G,x)=x^3-6x^2-60x+196$.\qed
\end{example}

Note that if $w(e)=1$ for all $e \in E(G)$ and $w_1(u)=0$ for all $u\in V(G)$, we recover the usual characteristic polynomial of $G$ and $B_{(w,w_1)}(G)$ is the usual adjacency matrix. Godsil and Gutman \cite[Theorem 4]{GG} first proved  the relation between the characteristic polynomial of $G$ and its matching polynomial. In Section 2, we shall show that similar relation holds for weighted characteristic polynomial and weighted matching polynomial (Theorem \ref{main_1}). As a consequence, the weighted characteristic polynomial of a graph and its the weighted matching polynomial are identical if and only if the graph is a forest, provided that the edge weight function $w$ is positive real-valued (Corollary \ref{tree_matching_characteristic_identical_converse}).

We would like to remark that `ordering' in $V(G)$ is very important. Different ordering in $V(G)$ could give different weighted characteristic polynomial (see Example \ref{example_ordering}). This also means that in general weight-isomorphic graphs might not have the same weighted characteristic polynomials. However if $G$ is a tree or $G$ is any graph with real valued edge weight function, then the `ordering' in $V(G)$ will have no effect on its weighted characteristic polynomial (Corollary \ref{ordering_tree_characteristic} and Corollary \ref{ordering_edge_weight_real_characteristic}, respectively).

\begin{example}\label{example_ordering} Let $G_1$ and $G_2$ be the graphs in Figure 4. Let $V(G_1)=\{u_1,u_2,u_3,u_4\}$ and $V(G_2)=\{v_1,v_2,v_3,v_4\}$. Suppose the vertex weight functions for both graphs take value 0 for all the vertices, whereas the edge weight functions are as given in the figure.

\begin{center}
\begin{pspicture}(0,0)(7,3)
\cnodeput(1, 1){1}{}
\cnodeput(2, 1){2}{}
\cnodeput(2, 2){3}{}
\cnodeput(1, 2){4}{}
\cnodeput(5, 1){5}{}
\cnodeput(6, 1){6}{}
\cnodeput(6, 2){7}{}
\cnodeput(5, 2){8}{}
\ncline{1}{2}
\ncline{2}{3}
\ncline{3}{4}
\ncline{4}{1}
\ncline{5}{6}
\ncline{6}{7}
\ncline{7}{8}
\ncline{5}{8}
\rput(3.5,0){\textnormal{Figure 4}.}
\rput(0,1.5){$G_1=$}
\rput(4,1.5){$G_2=$}
\rput(1.5,0.7){$i$}
\rput(1.5,2.3){$i$}
\rput(2.3,1.5){$i$}
\rput(0.7,1.5){$i$}
\rput(5.5,0.7){$i$}
\rput(5.5,2.3){$i$}
\rput(6.3,1.5){$i$}
\rput(4.7,1.5){$i$}
\rput(0.7,0.7){$u_1$}
\rput(2.3,0.7){$u_2$}
\rput(2.3,2.3){$u_3$}
\rput(0.7,2.3){$u_4$}
\rput(4.7,0.7){$v_1$}
\rput(6.3,0.7){$v_3$}
\rput(6.3,2.3){$v_2$}
\rput(4.7,2.3){$v_4$}
\end{pspicture}
\end{center}
Now order the vertices of $G_1$ as follows: $u_1\equiv 1$, $u_2\equiv 2$, $u_3\equiv 3$, $u_4\equiv 4$. Then
\begin{equation}
B_{(w,w_1)}(G_1)=\begin{pmatrix}
0 & i & 0 & i\\
-i & 0 & i & 0\\
0 & -i & 0 & i\\
-i & 0 & -i & 0
\end{pmatrix},\notag
\end{equation}
and $\phi_{(w,w_1)}(G_1,x)=x^4-4x^2+4$.

Suppose we order the vertices of $G_2$ as follows: $v_1\equiv 1$, $v_2\equiv 2$, $v_3\equiv 3$, $v_4\equiv 4$. Then
\begin{equation}
B_{(w,w_1)}(G_2)=\begin{pmatrix}
0 & 0 & i & i\\
0 & 0 & i & i\\
-i & -i & 0 & 0\\
-i & -i & 0 & 0
\end{pmatrix},\notag
\end{equation}
and $\phi_{(w,w_1)}(G_2,x)=x^4-4x^2$. So even though $G_1$ is weight-isomorphic to $G_2$, $\phi_{(w,w_1)}(G_1,x)\neq \phi_{(w,w_1)}(G_2,x)$.\qed
\end{example}

We shall denote the multiplicity of $\theta$ as a root of $\eta_{(w,w_1)}(G,x)$  and $\mu_w(G,x)$ by $\textnormal {mult} (\theta,G,\eta_{(w,w_1)})$ and $\textnormal {mult} (\theta,G,\mu_w)$ respectively. In Section 4, we classify the vertices of $G$ with respect to $\theta$ using Godsil's approach \cite[Section 3]{G2} and study their properties. In Section 5, we develop a Gallai-Edmonds decomposition associated to a root $\theta$ of the weighted matching polynomial (Corollary \ref{Gallai_Edmond_decomposition} and Corollary \ref{Gallai_critical}). In Section 6, we discuss the connection of our result with the classical Gallai-Edmonds decomposition which is associated to root $\theta=0$. In Section 7, we deduce the Parter-Weiner theorem and its generalization.

\section{Weighted matching polynomial and weighted characteristic polynomial}

It is not difficult to verify the following recurrence relations of $\mu_{w}(G,x)$ following the proof in \cite[Theorem 1.1 on p. 2]{G0}. The sketch of the proofs are provided.

\begin{lm}\label{pre_basic_recurrence} Recurrence for $\mu_{w}(G,x)$. ($v \sim u$ means $u$ is adjacent to $v$)
\begin{itemize}
\item [\textnormal{(a)}] $\mu_{w}(  G\cup  H,x)=\mu_{w}(  G,x)\mu_{w}(  H,x)$ where $  G$ and $  H$ are disjoint graphs.
\item [\textnormal{(b)}] $\mu_{w}(  G,x)=\mu_w (  G-e_{uv},x)-\vert w(e_{uv})\vert^2\mu_{w}(  G\setminus uv, x)$ if $e_{uv}$ is an edge of $G$.
\item [\textnormal{(c)}] $\mu_{w}(G,x) = x \mu_w(G \setminus u,x) - \sum_{v \sim u} \vert w(e_{uv})\vert^2 \mu_{w}(G \setminus uv,x)$.
\item [\textnormal{(d)}] $\frac{d}{dx}(\mu_{w}(G,x)) =\sum_{v\in V(G)}\mu_{w}(G \setminus v,x)$.
\end{itemize}
\end{lm}

\begin{proof} (a) Note that every $r$-matching in $G\cup H$ consists of an $s$-matching in $G$ and an $r-s$-matching in $H$. So for each $M\in M_{r}(G\cup H)$, $M=M_1\cup M_2$ for some $M_1\in M_s(G)$ and $M_2\in M_{r-s}(H)$. Part (a) follows from Lemma \ref{different_form_generalized_matching}, by noticing that
\begin {align}
\sum_{M \in M_r(G\cup H)} \vert w(M)\vert^2 &=\sum_{s=0}^{r} \sum_{M_1\in M_s(G), \atop M_2\in M_{r-s}(H)} \vert w(M_1\cup M_2)\vert^2\notag\\ &= \sum_{s=0}^{r}\left(\sum_{M_1\in M_s(G)}\vert w(M_1)\vert^2\right)\left(\sum_{M_2\in M_{r-s}(H)} \vert w(M_2)\vert^2\right ).\notag
\end {align}

\noindent
(b) Let $P_r(e_{uv})=\{ M\in M_r(G)\ :\ e_{uv}\in M\}$. Note that if $M\in P_r(e_{uv})$, then $M\setminus \{e_{uv}\}$ is an  $(r-1)$-matching in $G\setminus uv$, i.e. $M\setminus \{e_{uv}\}\in M_{r-1}(G\setminus uv)$. Also $M_r(G)\setminus P_r(e_{uv})=M_r(G-e_{uv})$. Thus $M_r(G)=P_r(e_{uv})\cup M_r(G-e_{uv})$. Part (b) follows from Lemma \ref{different_form_generalized_matching} by noticing that

\begin {align}
\sum_{M \in M_r(G)} \vert w(M)\vert ^2 &=\sum_{M\in M_r(G-e_{uv})} \vert w(M)\vert^2+\sum_{M\in P_r(e_{uv})} \vert w(M)\vert^2\notag\\
&=\sum_{M\in M_r(G-e_{uv})} \vert w(M)\vert ^2+\vert w(e_{uv})\vert^2\sum_{M\in M_{r-1}(G\setminus uv)} \vert w(M)\vert^2.\notag
\end {align}

\noindent
(c) Note that $M_r(G)=M_{r}(G\setminus u)\cup \left(\bigcup_{v \sim u} P_r(e_{uv})\right)$, where $P_r(e_{uv})=\{ M\in M_r(G)\ :\ e_{uv}\in M\}$. So part (c) follows from Lemma \ref{different_form_generalized_matching} by noticing that

\begin {align}
\sum_{M \in M_r(G)} \vert w(M)\vert ^2 &=\sum_{M\in M_{r}(G\setminus u)} \vert w(M)\vert ^2+\sum_{v \sim u}\sum_{M\in P_r(e_{uv})} \vert w(M)\vert ^2\notag\\
&=\sum_{M\in M_{r}(G\setminus u)} \vert w(M)\vert ^2+\sum_{v \sim u}\vert w(e_{uv})\vert ^2\sum_{M\in M_{r-1}(G\setminus uv)} \vert w(M)\vert^2.\notag
\end {align}

\noindent
(d) Let $\vert V(G)\vert=n$. Then
\begin{equation}
\frac{d}{dx}(\mu_{w}(G,x)) =\sum_{r=0}^{\lfloor (n-1)/2\rfloor}(-1)^r(n-2r)\left (\sum_{M \in M_r(G)} \vert w(M)\vert ^2\right) x^{n-1-2r}.\notag
\end{equation}
Let
\begin {equation}
T_r(G)=\{ (M,v)\in M_r(G)\times V(G)\ :\ v\ \text {is not contained in any of the edges in}\ M\}.\notag
\end {equation}
Let us calculate the sum $\sum_{(M,v)\in T_r(G)} \vert w(M)\vert ^2$ in two ways. First we fix $M$ and count the number of $v$. Since $M$ contains exactly $r$ edges and each of the edges contains exactly 2 vertices, the number of vertices that are not contained in any of the edges in $M$ is equal to $n-2r$. Therefore $\sum_{(M,v)\in T_r(G)} \vert w(M)\vert ^2=(n-2r)\left (\sum_{M \in M_r(G)} \vert w(M)\vert^2\right)$.

Second we fix $v$ and count the number of $M$. This is the number of $r$-matching in $G\setminus v$. Therefore $\sum_{(M,v)\in T_r(G)} \vert w(M)\vert ^2=\sum_{v\in V(G)}\sum_{M \in M_r(G\setminus v)} \vert w(M)\vert^2$. Part (d) then follows from Lemma \ref{different_form_generalized_matching}.
\end{proof}

\begin{thm}\label{basic_recurrence} Recurrence for $\eta_{(w,w_1)}(G,x)$. ($v \sim u$ means $u$ is adjacent to $v$)
\begin{itemize}
\item [\textnormal{(a)}] $\eta_{(w,w_1)}(  G_1\cup  G_2,x)=\eta_{(w,w_1)}(  G_1,x)\eta_{(w,w_1)}(G_2,x)$ where $  G_1$ and $G_2$ are disjoint graphs.
\item [\textnormal{(b)}] $\eta_{(w,w_1)}(  G,x)=\eta_{(w,w_1)} (  G-e_{uv},x)-\vert w(e_{uv})\vert^2\eta_{(w,w_1)}(  G\setminus uv, x)$ if $e_{uv}$ is an edge of $G$.
\item [\textnormal{(c)}] $\eta_{(w,w_1)}(G,x) = (x-w_1(u)) \eta_{(w,w_1)}(G \setminus u,x) - \sum_{v \sim u} \vert w(e_{uv})\vert^2 \eta_{(w,w_1)}(G \setminus uv,x)$.
\item [\textnormal{(d)}] $\frac{d}{dx}(\eta_{(w,w_1)}(G,x)) =\sum_{v\in V(G)}\eta_{(w,w_1)}(G \setminus v,x)$.
\end{itemize}
\end{thm}

\begin{proof} (a) For each $S\subseteq V(G_1\cup G_2)$, $S=S_1\cup S_2$ with $S_1\subseteq V(G_1)$ and $S_2\subseteq V(G_2)$. Also by part (a) of Lemma \ref{pre_basic_recurrence},  $\mu_{w} (H_{G_1\cup G_2}(S),x)=\prod_{j=1}^2\mu_{w} (H_{G_j}(S_j),x)$. Therefore
\begin{align}
(-1)^{\vert V((G_1\cup G_2)\setminus S)\vert}&w_1((G_1\cup G_2)\setminus S)\mu_{w} (H_{G_1\cup G_2}(S),x)\notag\\
&=\prod_{j=1}^2(-1)^{\vert V(G_j\setminus S_j)\vert}w_1(G_j\setminus S_j)\mu_{w} (H_{G_j}(S_j),x),\notag
\end{align}
and
\begin{align}
\eta_{(w,w_1)}(G_1\cup G_2,x) &=\sum_{S_1\subseteq V(G_1)}\sum_{S_2\subseteq V(G_2)}\prod_{j=1}^2(-1)^{\vert V(G_j\setminus S_j)\vert}w_1(G_j\setminus S_j)\mu_{w} (H_{G_j}(S_j),x)\notag\\
&=\eta_{(w,w_1)}(  G_1,x)\eta_{(w,w_1)}(G_2,x).\notag
\end{align}

\noindent
(b) Note that if $S\subseteq V(G)$ then either $\{u,v\}\subseteq S$ or $\{u,v\}\nsubseteq S$. Therefore
\begin{align}
 \eta_{(w,w_1)}(G,x) =\sum_{\substack{S\subseteq V(G),\\ \{u,v\}\subseteq S}}(-1)^{\vert V(G\setminus S)\vert}&w_1(G\setminus S)\mu_{w} (H_G(S),x)+\notag\\
 &\sum_{\substack{S\subseteq V(G),\\ \{u,v\}\nsubseteq S}}(-1)^{\vert V(G\setminus S)\vert}w_1(G\setminus S)\mu_{w} (H_G(S),x).\notag
 \end{align}
Now if $\{u,v\}\subseteq S$ then $e_{uv}\in E(H_G(S))$. So by part (b) of Lemma \ref{pre_basic_recurrence},
\begin{align}
\sum_{\substack{S\subseteq V(G),\\ \{u,v\}\subseteq S}}(-1)^{\vert V(G\setminus S)\vert}& w_1(G\setminus S)\mu_{w} (H_G(S),x)=\notag\\
&\left (\sum_{\substack{S\subseteq V(G),\\ \{u,v\}\subseteq S}}(-1)^{\vert V(G\setminus S)\vert}w_1(G\setminus S)\mu_{w} (H_G(S)-e_{uv},x)\right)-\notag\\
&\left (\vert w(e_{uv})\vert^2\sum_{\substack{S\subseteq V(G),\\ \{u,v\}\subseteq S}}(-1)^{\vert V(G\setminus S)\vert}w_1(G\setminus S)\mu_{w} (H_G(S)\setminus uv,x)\right).\notag
\end{align}

On the other hand, by setting $G'=G-e_{uv}$ we have
\begin{align}
 \eta_{(w,w_1)}(G',x) =\sum_{\substack{S\subseteq V(G'),\\ \{u,v\}\subseteq S}}(-1)^{\vert V(G'\setminus S)\vert}&w_1(G'\setminus S)\mu_{w} (H_{G'}(S),x)+\notag\\
 &\sum_{\substack{S\subseteq V(G'),\\ \{u,v\}\nsubseteq S}}(-1)^{\vert V(G'\setminus S)\vert}w_1(G'\setminus S)\mu_{w} (H_{G'}(S),x).\notag
 \end{align}
Note that $(-1)^{\vert V(G'\setminus S)\vert}w_1(G'\setminus S)=(-1)^{\vert V(G\setminus S)\vert}w_1(G\setminus S)$. Furthermore if $\{u,v\}\nsubseteq S$, then $H_{G'}(S)=H_{G}(S)$ and $\mu_{w} (H_{G'}(S),x)=\mu_{w} (H_G(S),x)$. Therefore
\begin{align}
\sum_{\substack{S\subseteq V(G'),\\ \{u,v\}\nsubseteq S}}(-1)^{\vert V(G'\setminus S)\vert}&w_1(G'\setminus S)\mu_{w} (H_{G'}(S),x)=\notag\\
&\sum_{\substack{S\subseteq V(G),\\ \{u,v\}\nsubseteq S}}(-1)^{\vert V(G\setminus S)\vert}w_1(G\setminus S)\mu_{w} (H_G(S),x).\notag
\end{align}
Also if $\{u,v\}\subseteq S$ then $H_{G'}(S)=H_G(S)-e_{uv}$. Therefore
\begin{align}
 \eta_{(w,w_1)}(G,x) = \eta_{(w,w_1)}&(G-e_{uv},x) -\notag\\
 &\left (\vert w(e_{uv})\vert^2\sum_{\substack{S\subseteq V(G),\\ \{u,v\}\subseteq S}}(-1)^{\vert V(G\setminus S)\vert}w_1(G\setminus S)\mu_{w} (H_G(S)\setminus uv,x)\right).\notag
 \end{align}
Now for each $S\subseteq V(G)$ and $\{u,v\}\subseteq S$, $S=S_1\cup \{u,v\}$ where $S_1\subseteq V(G\setminus uv)$.
Note also that $H_G(S)\setminus uv=H_{G\setminus uv} (S_1)$ and $G\setminus S=(G\setminus uv)\setminus S_1$. Hence  $\eta_{(w,w_1)}(  G,x)=\eta_{(w,w_1)} (  G-e_{uv},x)-\vert w(e_{uv})\vert^2\eta_{(w,w_1)}(  G\setminus uv, x)$.

\noindent
(c) Let $v_1,\dots, v_k$ be all the vertices adjacent to $u$ in $G$ and $g(S)=(-1)^{\vert V(G\setminus S)\vert}w_1(G\setminus S)\mu_{w} (H_G(S),x)$. For a set $T \subseteq \{v_{1}, \ldots, v_{k}\}$, let $N(T) = \{S \subseteq V(G) : u \in S~\textnormal{and} ~\textnormal{if}~ v \sim u~\textnormal{in}~H_{G}(S)~\textnormal{then}~v
\in T \}$. Then
\begin{equation}
 \eta_{(w,w_1)}(G,x) =\sum_{\substack{S\subseteq V(G),\\ u\notin S}}g(S)+\sum_{T\subseteq \{v_1,\dots,v_k\}}\sum_{S \in N(T)}g(S).\notag
 \end{equation}
For each $S\subseteq V(G)$ and $u\notin S$, we have $S\subseteq V(G\setminus u)$, and vice versa. Therefore $(-1)^{\vert V(G\setminus S)\vert}=-(-1)^{\vert V((G\setminus u)\setminus S)\vert}$, $w_1(G\setminus S)=w_1(u)w_1((G\setminus u)\setminus S)$ and $H_G(S)=H_{G\setminus u}(S)$. So
\begin{equation}
\sum_{\substack{S\subseteq V(G),\\ u\notin S}}g(S)=-w_1(u)\eta_{(w,w_1)}(G \setminus u,x).\notag
\end{equation}
Let $T\subseteq \{v_1,\dots,v_k\}$ (note that $T$ can be empty set). By part (c) of Lemma \ref{pre_basic_recurrence}, for each $S\subseteq V(G)$ such that $S \in N(T)$, we have
\begin{align}
g(S)=x(-1)^{\vert V(G\setminus S)\vert}&w_1(G\setminus S)\mu_{w} (H_G(S)\setminus u,x)-\notag\\
&\sum_{v\in T}\vert w(e_{uv})\vert^2(-1)^{\vert V(G\setminus S)\vert}w_1(G\setminus S)\mu_{w} (H_G(S)\setminus uv,x).\notag
\end{align}
Furthermore $S=\{u\}\cup S_1$ for some $S_1\subseteq G\setminus u$ and $T\subseteq S_1$. Also $(-1)^{\vert V(G\setminus S)\vert}=(-1)^{\vert V((G\setminus u)\setminus S_1)\vert}$, $w_1(G\setminus S)=w_1((G\setminus u)\setminus S_1)$ and $H_G(S)\setminus u=H_{G\setminus u}(S_1)$. When $S_{1}$ runs through all the subsets of $V(G \setminus u)$, $T$ runs through all the subsets of $\{v_1,\dots, v_k\}$. Therefore
\begin{align}
&x\left (\sum_{T\subseteq \{v_1,\dots,v_k\}}\sum_{ S \in N(T)}(-1)^{\vert V(G\setminus S)\vert}w_1(G\setminus S)\mu_{w} (H_G(S)\setminus u,x)\right)\notag\\
&=x\left (\sum_{S_1\subseteq V(G\setminus u)}(-1)^{\vert V((G\setminus u)\setminus S_1)\vert}w_1((G\setminus u)\setminus S_1)\mu_{w} (H_{(G\setminus u)}(S_1),x)\right)\notag\\
&=x\left (\eta_{(w,w_1)}(G \setminus u,x)\right).\notag
\end{align}
Also
\begin{align}
&\left (\sum_{T\subseteq \{v_1,\dots,v_k\}}\sum_{ S \in N(T)}\sum_{v\in T}\vert w(e_{uv})\vert^2(-1)^{\vert V(G\setminus S)\vert}w_1(G\setminus S)\mu_{w} (H_G(S)\setminus uv,x)\right)\notag\\
&=\sum_{v \sim u} \vert w(e_{uv})\vert^2 \left(\sum_{S_2\subseteq V(G\setminus uv)}(-1)^{\vert V((G\setminus uv)\setminus S_2)\vert}w_1((G\setminus uv)\setminus S_2)\mu_{w} (H_{(G\setminus uv)}(S_2),x)\right)\notag\\
&=\sum_{v \sim u} \vert w(e_{uv})\vert^2 \eta_{(w,w_1)}(G \setminus uv,x),\notag
\end{align}
where the first equality holds by comparing each term on the left and right sides of the equations: if $T=\varnothing$ then $\sum_{v\in T}\vert w(e_{uv})\vert^2(-1)^{\vert V(G\setminus S)\vert}w_1(G\setminus S)\mu_{w} (H_G(S)\setminus uv,x)=0$. So we may assume $T\neq \varnothing$. For a fixed $v\in T$, the term $\vert w(e_{uv})\vert^2(-1)^{\vert V(G\setminus S)\vert}w_1(G\setminus S)\mu_{w} (H_G(S)\setminus uv,x)$ is on the left side of the equation. Note that $S=S_2\cup \{u,v\}$ with $S_2\subseteq V(G\setminus uv)$. Also $(-1)^{\vert V(G\setminus S)\vert}=(-1)^{\vert V((G\setminus uv)\setminus S_2)\vert}$, $w_1(G\setminus S)=w_1((G\setminus uv)\setminus S_2)$ and $H_G(S)\setminus uv=H_{G\setminus uv}(S_2)$. Therefore
\begin{align}
\vert w(e_{uv})\vert^2(-1)^{\vert V(G\setminus S)\vert}&w_1(G\setminus S)\mu_{w} (H_G(S)\setminus uv,x)=\notag\\
&\vert w(e_{uv})\vert^2(-1)^{\vert V((G\setminus uv)\setminus S_2)\vert}w_1((G\setminus uv)\setminus S_2)\mu_{w} (H_{(G\setminus uv)}(S_2),x),\notag
\end{align}
which is a term on the right side of the equation. It is not hard to see that the terms on the left side is in one-to-one correspondence with the terms on the right.

Hence we have $\eta_{(w,w_1)}(G,x) = (x-w_1(u)) \eta_{(w,w_1)}(G \setminus u,x) - \sum_{v \sim u} \vert w(e_{uv})\vert^2 \eta_{(w,w_1)}(G \setminus uv,x)$.

\noindent
(d) Note that
\begin{align}
\frac{d}{dx}(\eta_{(w,w_1)}(G,x))&=\sum_{S\subseteq V(G)}(-1)^{\vert V(G\setminus S)\vert}w_1(G\setminus S)\frac{d}{dx}\mu_{w} (H_G(S),x)\notag\\
&=\sum_{S\subseteq V(G)}(-1)^{\vert V(G\setminus S)\vert}w_1(G\setminus S)\sum_{v\in S}\mu_{w} (H_G(S)\setminus v,x),\notag
\end{align}
where the second equality follows from part (d) of Lemma \ref{pre_basic_recurrence}. Note that
\begin{align}
\sum_{S\subseteq V(G)}(-1)^{\vert V(G\setminus S)\vert}&w_1(G\setminus S)\sum_{v\in S}\mu_{w} (H_G(S)\setminus v,x)\notag\\
&=\sum_{v\in V(G)}\left ( \sum_{S_1\subseteq V(G\setminus v)}(-1)^{\vert V((G\setminus v)\setminus S_1)\vert}w_1((G\setminus v)\setminus S_1)\mu_{w} (H_{G\setminus v}(S_1),x)\right)\notag\\
&=\sum_{v\in V(G)}\eta_{(w,w_1)}(G \setminus v,x),\notag
\end{align}
where the first equality holds by comparing each term on the left and right sides of the equations: for a fixed $S$ and $v\in S$, the term $(-1)^{\vert V(G\setminus S)\vert}w_1(G\setminus S)\mu_{w} (H_G(S)\setminus v,x)$ is on the left side of the equation. Note that $S=S_1\cup \{v\}$ with $S_1\subseteq V(G\setminus v)$. Also $(-1)^{\vert V(G\setminus S)\vert}=(-1)^{\vert V((G\setminus v)\setminus S_1)\vert}$, $w_1(G\setminus S)=w_1((G\setminus v)\setminus S_1)$ and $H_G(S)\setminus v=H_{G\setminus v}(S_1)$.  Therefore
\begin{align}
(-1)^{\vert V(G\setminus S)\vert}&w_1(G\setminus S)\mu_{w} (H_G(S)\setminus v,x)=\notag\\
&(-1)^{\vert V((G\setminus v)\setminus S_1)\vert}w_1((G\setminus v)\setminus S_1)\mu_{w} (H_{(G\setminus v)}(S_1),x),\notag
\end{align}
which is a term on the right side of the equation. It is not hard to see that the terms on the left side is in one-to-one correspondence with the terms on the right. Hence the proof is completed.
\end{proof}

\begin{dfn}\label{elementary_graph} An \emph {elementary} graph is a disjoint union of single edges ($K_2$) or cycles ($C_r$).

\noindent
A \emph {spanning elementary subgraph} of a graph is an elementary subgraph that contains all the vertices of the graph.

\noindent
We denote $\text {comp} (G)$ as the number of components in $G$.

\noindent
For convenience, we shall write $w(H)=\prod_{e\in E(H)} w(e)$ for any subgraph $H$ of $G$.
\end{dfn}

\noindent
Let $V(G)=\{1,2,\dots, n\}$. Let $v_1v_2\dots v_mv_1$, $m\geq 3$ be a cycle $C$ in $G$. We set
\begin{equation}
w_2(C)=b_{v_1v_2}b_{v_2v_3}\dots b_{v_{m-1}v_m}b_{v_mv_1}+b_{v_1v_m}b_{v_mv_{m-1}}\dots b_{v_{3}v_2}b_{v_2v_1},\notag
\end{equation}
where $b_{uv}$ is the $uv$-entry in the weighted adjacency matrix $B_{(w,w_1)}(G)$. Note that  $w_2(C)=b+\overline b$ where $b=b_{v_1v_2}b_{v_2v_3}\dots b_{v_{m-1}v_m}b_{v_mv_1}$. So $w_2(C)$ is a real number. The following lemma is obvious.

\begin{lm}\label{real_part_weight_function}  If the edge weight function $w$ is positive real-valued, then $w_2(C)>0$ for any cycle $C$ in $G$.\qed
\end{lm}

Now let us extend $w_2$ to the union of disjoint cycles. Let $C_1, C_2,\dots ,C_k$ be disjoint cycles in $G$ and $C=C_1\cup C_2\cup\cdots\cup C_k$. We set $w_2(C)=\prod_{j=1}^k w_2(C_j)$. We are ready to prove the next lemma whose non-weighted version was first observed by Harary \cite[Proposition 7.2]{B}.

\begin{lm}\label{determinant} Suppose $w_1(u)=0$ for all $u\in V(G)$. Let $\Gamma$ be the set of all spanning elementary subgraphs of $G$ and $\vert V(G)\vert=n$. Then
\begin{equation}
\det B_{(w,w_1)}(G)=(-1)^n\sum_{\gamma\in \Gamma} (-1)^{\textnormal {com} (\gamma)}\vert w(\gamma\setminus C_{\gamma})\vert ^2w_2(C_{\gamma}),\notag
\end{equation}
where $C_{\gamma}$ is the union of all the cycles in $\gamma$. In particular $\det B_{(w,w_1)}(G)$ is a real number.
\end{lm}

\begin{proof} Let $B_{(w,w_1)}(G)=[b_{uv}]$. Recall that $\det B_{(w,w_1)}(G)=\sum_{\pi\in S_n} \textnormal {sign} (\pi) \prod_{u=1}^n b_{u\pi(u)}$ (see \cite[Definition 1.2.2 on p. 6]{Mir}) where $S_n$ is the set of all permutations on $V(G)=\{1,2,\dots, n\}$.  Note that $b_{uu}=w_1(u)=0$. So the term $\prod_{u=1}^n b_{u\pi(u)}$ vanishes if $b_{u\pi(u)}=0$ for some $u$, that is either $\pi(u)=u$, or $\pi(u)\neq u$ and $e_{u\pi(u)}$ is not an edge in $G$. Therefore each non-vanishing term corresponds to a disjoint union of edges and cycles, which is a spanning subgraph of $G$. Furthermore the $\pi$ that corresponds to the the non-vanishing term can be written as a product of disjoint cycles of length at least 2 which is actually in correspondence to a spanning elementary subgraph of $G$ (the fact that every $\pi\in S_n$ can be written as a product of disjoint cycles can be found in \cite[Exercise 1.2.5 on p. 3]{Dix}).

Let $S\subseteq S_n$ be the set of all $\pi$ for which $\prod_{u=1}^n b_{u\pi(u)}\neq 0$. Let $f:S\rightarrow \Gamma$ be defined by $f(\pi)=\gamma$ where $\gamma$ is the spanning elementary subgraph corresponds to $\pi$. Let $\gamma\in\Gamma$. First let us find $\prod_{u=1}^n b_{u\pi(u)}$ for each $\pi\in f^{-1}(\gamma)$. Let $\pi\in f^{-1}(\gamma)$.  Let $u_1u_2$ be an edge ($K_2$) in $\gamma$. Then in the decomposition of $\pi$, it must have the cycle $(u_1\ u_2)$. Let $v_1v_2v_3\dots v_{m-1}v_mv_1$, $m\geq 3$ be a cycle in $\gamma$. Then in the decomposition of $\pi$, it must have either the cycle $(v_1\ v_2\ v_3 \dots v_m)$ or $(v_1\ v_m\ v_{m-1}\dots v_2)$. Note that $(v_1\ v_2\ v_3 \dots v_m)^{-1}=(v_1\ v_m\ v_{m-1}\dots v_2)$.

Let $\pi_{\gamma}\in f^{-1}(\gamma)$ be fixed. Then $\pi_{\gamma}=\tau_1'\tau_2'\dots \tau_{k_1}'\tau_1\tau_2\dots\tau_{k_2}$ where $\tau_j'$ is a 2-cycle and $\tau_j$ is a $m_j$-cycle, $m_j\geq 3$. For each $\pi\in f^{-1}(\gamma)$, $\pi=\tau_1'\tau_2'\dots \tau_{k_1}'\tau_1^{\pm 1}\tau_2^{\pm 1}\dots\tau_{k_2}^{\pm 1}$. Therefore $\textnormal {sign} (\pi)=\textnormal {sign} (\pi_{\gamma})$ and $\vert f^{-1}(\gamma)\vert =2^{k_2}$.

Suppose $\tau_1'=(u_1\ u_2)$ (we may assume $u_1<u_2$). Then $b_{u_1u_2}b_{u_2u_1}=w(e_{u_1u_2})\overline{w(e_{u_1u_2})}=\vert w(e_{u_1u_2})\vert^2$ is a term in $\prod_{u=1}^n b_{u\pi(u)}$.  Note that $\gamma\setminus C_{\gamma}$ consists of the union of $k_1$ edges ($K_2$) and each of these edges correspond to a $\tau_j'$. Therefore for each $\pi\in f^{-1}(\gamma)$, $\vert w(\gamma\setminus C_{\gamma})\vert^2$ is a term in $\prod_{u=1}^n b_{u\pi(u)}$.

Suppose $\tau_1=(v_1\ v_2\ v_3 \dots v_m)$. Then $b_{v_1v_2}b_{v_2v_3}\dots b_{v_{m-1}v_m}b_{v_mv_1}$ is a term in $\prod_{u=1}^n b_{u\pi(u)}$.  Note that $C_{\gamma}$ consists of the union of $k_2$ cycles and each of these cycles correspond to a $\tau_j$. Therefore if we sum up all the $\pi\in f^{-1}(\gamma)$, we have
\begin{equation}
\sum_{\pi\in f^{-1}(\gamma)} \textnormal {sign} (\pi)\prod_{u=1}^n b_{u\pi(u)}=\textnormal {sign} (\pi_{\gamma}) \vert w(\gamma\setminus C_{\gamma})\vert ^2w_2(C_{\gamma}).\notag
\end{equation}

Now let us find $\textnormal {sign} (\pi_{\gamma})$.  A cycle in $\gamma$ is called an even cycle if it contains odd number of vertices and an odd cycle otherwise. A $K_2$ in $\gamma$ is also called an odd cycle. Let the number of even cycles and the number of odd cycles in $\gamma$ be $N_e$ and $N_o$ respectively. Then $n\equiv N_e\mod 2$. Now $\textnormal {sign} (\pi_{\gamma})=(-1)^{N_o}$. Since $\textnormal {comp} (\gamma)=N_o+N_e$, we conclude that $\textnormal {sign} (\pi_{\gamma})=(-1)^{\textnormal {comp} (\gamma)+n}$. Therefore
\begin{equation}
\sum_{\pi\in f^{-1}(\gamma)} \textnormal {sign} (\pi)\prod_{u=1}^n b_{u\pi(u)}=(-1)^{\textnormal {comp} (\gamma)+n} \vert w(\gamma\setminus C_{\gamma})\vert ^2w_2(C_{\gamma}).\notag
\end{equation}

Hence
\begin{align}
\det B_{(w,w_1)}(G)& =\sum_{\pi\in S} \textnormal {sign} (\pi) \prod_{u=1}^n b_{u\pi(u)}\notag\\
& =\sum_{\gamma\in\Gamma} \sum_{\pi\in f^{-1}(\gamma)}  \textnormal {sign} (\pi) \prod_{u=1}^n b_{u\pi(u)}\notag\\
& =\sum_{\gamma\in\Gamma} (-1)^{\textnormal {comp} (\gamma)+n} \vert w(\gamma\setminus C_{\gamma})\vert ^2w_2(C_{\gamma}).\notag
\end{align}
\end{proof}

We shall need the following theorem from matrix theory.

\begin{thm}\label{determinant_characteristic_matrix} \textnormal{(\cite[Theorem 7.1.2 on p. 197]{Mir})} Let $B$ be a $n\times n$ matrix. Then
\begin {equation}
\det (xI_n-B)=x^n+\sum_{k=0}^{n-1} (-1)^{n-k}\sum_{1\leq u_1<\cdots<u_k\leq n}\vert B(u_1,\dots ,u_k;u_1,\dots ,u_k)\vert x^{k},\notag
\end {equation}
where  $B(u_1,\dots, u_k;u_1,\dots,u_k)$ is the matrix obtained from $B$ by deleting the $u_1,\dots,u_k$ rows and $u_1,\dots,u_k$ columns. Note that  $B(u_1,\dots, u_k;u_1,\dots,u_k)$ is a $(n-k)\times (n-k)$ matrix.\qed
\end{thm}

\begin{lm}\label{determinant_characteristic_spanning} Suppose $w_1(u)=0$ for all $u\in V(G)$. Let $\Gamma_i$ be the set of all elementary subgraphs of $G$ with $n-i$ vertices and $\phi_{(w,w_1)} (G,x)=\sum_{r=0}^n c_rx^r$, where $n=\vert V(G)\vert$. Then $c_n=1$ and for $i=0,\dots, n-1$,
\begin {equation}
c_{i}=\sum_{\gamma\in\Gamma_i}(-1)^{\textnormal {comp} (\gamma)}\vert w(\gamma\setminus C_{\gamma})\vert^2w_2(C_{\gamma}),\notag
\end {equation}
where $C_{\gamma}$ is the union of all the cycles in $\gamma$. In particular $c_{n-1}=0$.
\end{lm}

\begin{proof} Let $B=B_{(w,w_1)}(G)$ and $V(G)=\{1,2,\dots, n\}$. By Theorem \ref{determinant_characteristic_matrix},
\begin {equation}
\phi_{(w, w_{1})} (G,x)=\det (xI_n-B)=x^n+\sum_{k=0}^{n-1} (-1)^{n-k}\sum_{1\leq u_1<\cdots<u_k\leq n}\vert B(u_1,\dots ,u_k;u_1,\dots ,u_k)\vert x^{k}.\notag
\end {equation}
If $H(u_1,\dots ,u_k)=G\setminus u_1\dots u_k$ then $B_{(w,w_1)}(H(u_1,\dots ,u_k))=B(u_1,\dots ,u_k;u_1,\dots ,u_k)$.
By Lemma \ref{determinant},
\begin{equation}
\sum_{1\leq u_1<\cdots<u_k\leq n}\vert B(u_1,\dots ,u_k;u_1,\dots ,u_k)\vert=(-1)^{n-k}\sum_{\gamma\in \Gamma_k} (-1)^{\textnormal {comp} (\gamma)}\vert w(\gamma\setminus C_{\gamma})\vert^2w_2(C_{\gamma}).\notag
\end{equation}
Therefore
\begin{equation}
\phi_{(w,w_1)}(G,x)=x^n+\sum_{k=0}^{n-1} \left(\sum_{\gamma\in \Gamma_k} (-1)^{\textnormal {comp} (\gamma)}\vert w(\gamma\setminus C_{\gamma})\vert^2w_2(C_{\gamma}) \right)x^{k}.\notag
\end{equation}
Hence the lemma holds. Finally $c_{n-1}=0$ because $\Gamma_{n-1}$ is the empty set.
\end{proof}

\begin{lm}\label{pre_main_1} Suppose $w_1(u)=0$ for all $u\in V(G)$. Let $\Gamma(c)$ be the set of all elementary subgraphs of $G$ which contains only cycles. Then
\begin {equation}
\phi_{(w,w_1)}(G,x)=\mu_w (G,x)+\sum_{C\in \Gamma (c)}(-1)^{\textnormal {comp}(C)}w_2(C)\mu_w (G\setminus C,x).\notag
\end {equation}
In particular $\phi_{(w,w_1)}(G,x)$ is a polynomial over the field of real number $\mathbb R$.
\end{lm}

\begin{proof} Let $\vert V(G)\vert=n$ and $\phi_w(G,x)=\sum_{r=0}^n c_rx^r$. By Lemma \ref{determinant_characteristic_spanning},  $c_n=1$ and for $i=0,\dots, n-1$, $c_{i}=\sum_{\gamma\in\Gamma_i}(-1)^{\textnormal {comp} (\gamma)}\vert w(\gamma\setminus C_{\gamma})\vert^2w_2(C_{\gamma})$,  where $\Gamma_i$ is the set of all elementary subgraphs of $G$ with $n-i$ vertices and
$C_{\gamma}$ is the union of all the cycles in $\gamma$.  Also $c_{n-1}=0$.

Let $\Gamma_i(1)=\{\gamma\in\Gamma_i\ :\ \gamma\ \text {does not contain any cycle}\}$ and $\Gamma_i(2)=\Gamma_i\setminus \Gamma_i(1)$. Let $g(\gamma)=(-1)^{\textnormal {comp} (\gamma)}\vert w(\gamma\setminus C_{\gamma})\vert^2w_2(C_{\gamma})$. Then
\begin{equation}
\phi_{(w, w_{1})}(G,x)=x^n+\sum_{r=0}^{n-2}  \sum_{\gamma\in\Gamma_r(1)}g(\gamma)x^r+
\sum_{r=0}^{n-2} \sum_{\gamma\in\Gamma_r(2)}g(\gamma)x^r.\notag
\end{equation}

Note that
\begin{equation}
\sum_{r=0}^{n-2} \sum_{\gamma\in\Gamma_r(1)}g(\gamma)x^r=\sum_{r=2}^{n} \sum_{\gamma\in\Gamma_{n-r}(1)}g(\gamma)x^{n-r}.\notag
\end{equation}
Now if $\gamma\in \Gamma_{n-r}(1)$ then $C_{\gamma}=\varnothing$ and $\textnormal {comp} (\gamma)$ is the number of $K_2$ in $\gamma$. Therefore $\vert w(\gamma\setminus C_{\gamma})\vert^2w_2(C_{\gamma})=\vert w(\gamma)\vert^2$, $\gamma$ is a $(r/2)$-matching in $G$ and the number of vertices in $\gamma$ is $r=2\textnormal {comp} (\gamma)$. This means that if $r$ is not even then the coefficient of $x^{n-r}$ is zero. Furthermore if $\gamma,\gamma'\in \Gamma_{n-r}(1)$ then $\textnormal {comp} (\gamma)=\textnormal {comp} (\gamma')$.  Let $d=\textnormal {comp} (\gamma)$. Then
\begin{align}
\sum_{r=2}^{n} \sum_{\gamma\in\Gamma_{n-r}(1)}g(\gamma)x^{n-r}& =\sum_{r=2}^{n} \sum_{\gamma\in\Gamma_{n-r}(1)}(-1)^{\textnormal {comp} (\gamma)}\vert w(\gamma)\vert^2x^{n-r}\notag\\
&=\sum_{d=1}^{\lfloor n/2\rfloor} (-1)^{d}\left (\sum_{M \in M_d(G)} \vert w(M)\vert^2\right) x^{n-2d}\notag
\end{align}
and by Lemma \ref{different_form_generalized_matching},
\begin{equation}
x^n+\sum_{r=0}^{n-2} \sum_{\gamma\in\Gamma_r(1)}g(\gamma)x^r=\mu_w(G,x).\notag
\end{equation}

Next
\begin{equation}
\sum_{r=0}^{n-2} \sum_{\gamma\in\Gamma_r(2)}g(\gamma)x^r=\sum_{r=2}^{n} \sum_{\gamma\in\Gamma_{n-r}(2)}g(\gamma)x^{n-r}.\notag
\end{equation}
For each $\gamma\in \Gamma_{n-r}(2)$, $C_{\gamma}\in\Gamma(c)$. We shall partition $\Gamma_{n-r}(2)$ according to $C\in\Gamma(c)$. Let
\begin{equation}
\Gamma_{n-r}(2)(C)=\{\gamma\in\Gamma_{n-r}(2)\ :\ \gamma\ \text {contains}\ C\ \text {and}\ \gamma\setminus C\ \text {is a disjoint union of}\ K_2\}.\notag
\end{equation}
Then $\{\Gamma_{n-r}(2)(C)\}_{C\in \Gamma(c)}$ is a partition for $\Gamma_{n-r}(2)$. If $\gamma\in \Gamma_{n-r}(2)(C)$ then $\textnormal {comp} (\gamma)=\textnormal {comp} (C)+\textnormal {comp} (\gamma\setminus C)$ and
the number of vertices in $\gamma$ is $r=2\textnormal {comp} (\gamma\setminus C)+\vert V(C)\vert$. This means that if $r\not\equiv \vert V(C)\vert\mod 2$ then the coefficient of $x^{n-r}$ is zero. Furthermore if $\gamma,\gamma'\in \Gamma_{n-r}(2)(C)$ then $\textnormal {comp} (\gamma\setminus C)=\textnormal {comp} (\gamma'\setminus C)$. So, $\gamma\setminus C$ and $\gamma'\setminus C$ are $((r-\vert V(C)\vert)/2)$-matching in $G\setminus C$. Let $d=\textnormal {comp} (\gamma\setminus C)$. Then by Lemma \ref{different_form_generalized_matching},
\begin{align}
\sum_{r=2}^{n} &\sum_{\gamma\in\Gamma_{n-r}(2)}g(\gamma)x^{n-r}\notag\\
&=\sum_{r=2}^{n} \sum_{\gamma\in\Gamma_{n-r}(2)}(-1)^{\textnormal {comp} (\gamma)}\vert w(\gamma\setminus C_{\gamma})\vert^2w_2(C_{\gamma})x^{n-r}\notag\\
 & =\sum_{r=2}^{n} \sum_{C\in\Gamma(c)} \sum_{\gamma\in\Gamma_{n-r}(2)(C)}(-1)^{\textnormal {comp} (\gamma)}\vert w(\gamma\setminus C_{\gamma})\vert^2w_2(C_{\gamma})x^{n-r}\notag\\
& =\sum_{r=2}^{n} \sum_{C\in\Gamma(c)} \sum_{\gamma\in\Gamma_{n-r}(2)(C)}(-1)^{\textnormal {comp} (C)+\textnormal {comp} (\gamma\setminus C)}\vert w(\gamma\setminus C_{\gamma})\vert^2w_2(C_{\gamma})x^{n-r}\notag\\
& =\sum_{r=2}^{n} \sum_{C\in\Gamma(c)} (-1)^{\textnormal {comp} (C)}w_2(C)\sum_{\gamma\in\Gamma_{n-r}(2)(C)}(-1)^{\textnormal {comp} (\gamma\setminus C)}\vert w(\gamma\setminus C)\vert^2x^{n-r}\notag\\
& =\sum_{C\in\Gamma(c)} (-1)^{\textnormal {comp} (C)}w_2(C)\sum_{r=2}^{n} \sum_{\gamma\in\Gamma_{n-r}(2)(C)}(-1)^{\textnormal {comp} (\gamma\setminus C)}\vert w(\gamma\setminus C)\vert^2x^{n-r}\notag\\
& =\sum_{C\in\Gamma(c)} (-1)^{\textnormal {comp} (C)}w_2(C)\sum_{d=0}^{\lfloor (n-\vert V(C)\vert)/2\rfloor} (-1)^{d}\left (\sum_{M \in M_d(G\setminus C)} \vert w(M)\vert^2\right)x^{n-\vert V(C)\vert-2d}\notag\\
& =\sum_{C\in\Gamma(c)} (-1)^{\textnormal {comp} (C)}w_2(C)\mu_w(G\setminus C,x).\notag
\end{align}
Hence the theorem holds.
\end{proof}

We wish to show that similar equation (Lemma \ref{pre_main_1}) holds even when $w_1(u)\neq 0$ for some $u\in V(G)$.  This will be done in Theorem \ref{main_1}. Before we do that, let us first prove Lemma \ref{lemma_before_main_1}.

\begin{lm}\label{lemma_before_main_1} Let $u\in V(G)$. Let $G_1$ be a graph isomorphic to $G$. We shall assume $V(G_1)=V(G)$, $E(G_1)=E(G)$ and the weight function $(t, t_1)$ on $G_1$ is defined by $t(e_{vv'})=w(e_{vv'})$ for all $e_{vv'}\in E(G_1)$, $t_1(u)=0$ and $t_1(v)=w_1(v)$ for all $v\in V(G_1)\setminus\{u\}$.

Let $G_2=G\setminus u$. Then
\begin{equation}
\eta_{(w,w_1)}(G,x) =\eta_{(t,t_1)}(G_1,x)-w_1(u)\eta_{(w,w_1)}(G_2,x).\notag
\end{equation}
\end{lm}

\begin{proof} Note that
\begin{align}
 \eta_{(w,w_1)}(G,x) =\sum_{\substack{S\subseteq V(G),\\ u\in S}}(-1)^{\vert V(G\setminus S)\vert}&w_1(G\setminus S)\mu_{w} (H_G(S),x)+\notag\\
 &\sum_{\substack{S\subseteq V(G),\\ u\notin S}}(-1)^{\vert V(G\setminus S)\vert}w_1(G\setminus S)\mu_{w} (H_G(S),x).\notag
 \end{align}
For each $S\subseteq V(G)$ with $u\in S$, $(-1)^{\vert V(G\setminus S)\vert}=(-1)^{\vert V(G_1\setminus S)\vert}$,
$w_1(G\setminus S)=t_1(G_1\setminus S)$ and $H_G(S)=H_{G_1}(S)$. Therefore
\begin{align}
 \sum_{\substack{S\subseteq V(G),\\ u\in S}}(-1)^{\vert V(G\setminus S)\vert}&w_1(G\setminus S)\mu_{w} (H_G(S),x)=\notag\\
 &\sum_{\substack{S\subseteq V(G_1),\\ u\in S}}(-1)^{\vert V(G_1\setminus S)\vert}t_1(G_1\setminus S)\mu_{t} (H_{G_1}(S),x).\notag
 \end{align}
 By Lemma \ref{weighted_matching_with_zero_vertex_weight}, $\eta_{(t,t_1)}(G_1,x)=\sum_{S\subseteq V(G_1), u\in S}(-1)^{\vert V(G_1\setminus S)\vert}t_1(G_1\setminus S)\mu_{t} (H_{G_1}(S),x)$.

For each $S\subseteq V(G)$ with $u\notin S$, $(-1)^{\vert V(G\setminus S)\vert}=-(-1)^{\vert V(G_2\setminus S)\vert}$, $w_1(G\setminus S)=w_1(u)w_1(G_2\setminus S)$ and $H_G(S)=H_{G_2}(S)$. Therefore
\begin{align}
 \sum_{\substack{S\subseteq V(G),\\ u\notin S}}(-1)^{\vert V(G\setminus S)\vert}&w_1(G\setminus S)\mu_{w} (H_G(S),x)\notag\\
 &=-w_1(u)\sum_{\substack{S\subseteq V(G_2)}}(-1)^{\vert V(G_2\setminus S)\vert}w_1(G_2\setminus S)\mu_{w} (H_{G_2}(S),x)\notag\\
 &=-w_1(u)\eta_{(w,w_1)}(G_2,x),\notag
 \end{align}
and $ \eta_{(w,w_1)}(G,x) =\eta_{(t,t_1)}(G_1,x)-w_1(u)\eta_{(w,w_1)}(G_2,x)$.
\end{proof}

\begin{thm}\label{main_1} Let $\Gamma(c)$ be the set of all elementary subgraphs of $G$ which contains only cycles. Then
\begin {equation}
\phi_{(w,w_1)}(G,x)=\eta_{(w,w_1)} (G,x)+\sum_{C\in \Gamma (c)}(-1)^{\textnormal {comp}(C)}w_2(C)\eta_{(w,w_1)} (G\setminus C,x).\notag
\end {equation}
In particular $\phi_{(w,w_1)}(G,x)$ is a polynomial over the field of real number $\mathbb R$.
\end{thm}

\begin{proof} Let $V(G)=\{1,2,\dots, n\}$. Let the number of non-zero in the sequence $w_1(1),w_1(2),\dots ,w_1(n)$ be denoted by $\kappa(G)$. We shall prove by induction on $\kappa(G)$. If $\kappa (G)=0$, that is $w_1(j)=0$ for all $j$, then the theorem holds (Lemma \ref{weighted_to_edge_weighted_matching} and Lemma \ref{pre_main_1}). Suppose $\kappa (G)>0$. Assume that the theorem holds for all graph $G'$ with $\kappa(G')<\kappa(G)$.

For convenience, we shall assume $w_1(1)\neq 0$ (similar argument can be used if $w_1(u)\neq 0$ for other $u$). Note that
\begin{equation}
\phi_{(w,w_1)}(G,x)=\det
\begin{pmatrix}
x-w_1(1) & -w(e_{12}) & \dots & -w(e_{1n})\\
-\overline {w(e_{12})} & x-w_2(2)& \dots& -w(e_{2n})\\
\vdots & \vdots &\ddots  &\vdots\\
-\overline {w(e_{1n})} & -\overline {w(e_{2n})}& \dots& x-w_1(n)
\end{pmatrix}.\notag
\end{equation}
So by Theorem 1.2.5 on p. 10 of \cite{Mir},
\begin{align}
\phi_{(w,w_1)}(G,x)=\det &
\begin{pmatrix}
x & -w(e_{12}) & \dots & -w(e_{1n})\\
-\overline {w(e_{12})} & x-w_1(2)& \dots& -w(e_{2n})\\
\vdots & \vdots &\ddots  &\vdots\\
-\overline {w(e_{1n})} & -\overline {w(e_{2n})}& \dots& x-w_1(n)
\end{pmatrix}+\notag\\
&\det \begin{pmatrix}
-w_1(1) & -w(e_{12}) & \dots & -w(e_{1n})\\
0 & x-w_1(2)& \dots& -w(e_{2n})\\
\vdots & \vdots &\ddots  &\vdots\\
0 & -\overline {w(e_{2n})}& \dots& x-w_1(n)
\end{pmatrix}.\notag
\end{align}
Let $G_1$ be a graph isomorphic to $G$. We may assume $V(G_1)=V(G)$ and $E(G_1)=E(G)$. Now let us define the weight function $(t, t_1)$ on $G_1$. Set $t(e_{uv})=w(e_{uv})$ for all $e_{uv}\in E(G_1)$, $t_1(1)=0$ and $t_1(j)=w_1(j)$ for all $j\geq 2$. Then
\begin{align}
\phi_{(t,t_1)}(G_1,x)=\det
\begin{pmatrix}
x & -w(e_{12}) & \dots & -w(e_{1n})\\
-\overline {w(e_{12})} & x-w_1(2)& \dots& -w(e_{2n})\\
\vdots & \vdots &\ddots  &\vdots\\
-\overline {w(e_{1n})} & -\overline {w(e_{2n})}& \dots& x-w_1(n)\notag
\end{pmatrix}
\end{align}
and by induction  (for $\kappa (G_1)<\kappa(G)$),
\begin{equation}
\phi_{(t,t_1)}(G_1,x)=\eta_{(t,t_1)} (G_1,x)+\sum_{C\in \Gamma (c)}(-1)^{\textnormal {comp}(C)}w_2(C)\eta_{(t,t_1)} (G_1\setminus C,x).\notag
\end{equation}

Let $G_2=G\setminus 1$. Then
\begin{align}
\phi_{(w,w_1)}(G_2,x)&=\det
\begin{pmatrix}
 x-w_1(2)& \dots& -w(e_{2n})\\
\vdots &\ddots  &\vdots\\
-\overline {w(e_{2n})}& \dots& x-w_1(n)\notag
\end{pmatrix}
\end{align}
and by induction  (for $\kappa (G_2)<\kappa(G)$),
\begin{equation}
\phi_{(w,w_1)}(G_2,x)=\eta_{(w,w_1)} (G_2,x)+\sum_{C\in \Gamma_2 (c)}(-1)^{\textnormal {comp}(C)}w_2(C)\eta_{(w,w_1)} (G_2\setminus C,x),\notag
\end{equation}
where $\Gamma_2 (c)$ is the set of all elementary subgraphs of $G_2$ which contains only cycles.

Note that $\phi_{(w,w_1)}(G,x)=\phi_{(t,t_1)}(G_1,x)-w_1(1)\phi_{(w,w_1)}(G_2,x)$ and by Lemma \ref{lemma_before_main_1}, $ \eta_{(w,w_1)}(G,x) =\eta_{(t,t_1)}(G_1,x)-w_1(1)\eta_{(w,w_1)}(G_2,x)$.

Next note that
\begin{align}
\sum_{C\in \Gamma (c)}(-1)^{\textnormal {comp}(C)}&w_2(C)\eta_{(t,t_1)} (G_1\setminus C,x)\notag=\\
 &\sum_{\substack {C\in \Gamma (c)\\ 1\in C}}(-1)^{\textnormal {comp}(C)}w_2(C)\eta_{(t,t_1)} (G_1\setminus C,x)+\notag\\
 &~~~~~\sum_{\substack {C\in \Gamma (c)\\ 1\notin C}}(-1)^{\textnormal {comp}(C)}w_2(C)\eta_{(t,t_1)} (G_1\setminus C,x).\notag
 \end{align}
For each $C\in \Gamma (c)$ with $1\in C$, we have $G_1\setminus C=G\setminus C$ (including the weight functions induced by it on the remaining vertices and edges in $G\setminus C$). Therefore $\eta_{(t,t_1)} (G_1\setminus C,x)=\eta_{(w,w_1)} (G\setminus C,x)$. For each $C\in \Gamma (c)$ with $1\notin C$, we have $\eta_{(w,w_1)} (G\setminus C,x)=\eta_{(t,t_1)} (G_1\setminus C,x)-w_1(1)\eta_{(w,w_1)} (G_2\setminus C,x)$ (Lemma \ref{lemma_before_main_1}). Therefore
\begin{align}
\sum_{\substack {C\in \Gamma (c)\\ 1\in C}}(-1)^{\textnormal {comp}(C)}&w_2(C)\eta_{(w,w_1)} (G\setminus C,x)=\notag\\
&\sum_{\substack {C\in \Gamma (c)\\ 1\in C}}(-1)^{\textnormal {comp}(C)}w_2(C)\eta_{(t,t_1)} (G_1\setminus C,x),\notag
\end{align}
and
\begin{align}
\sum_{\substack {C\in \Gamma (c)\\ 1\notin C}}(-1)^{\textnormal {comp}(C)}&w_2(C)\eta_{(w,w_1)} (G\setminus C,x)\notag=\\
 &\sum_{\substack {C\in \Gamma (c)\\ 1\notin C}}(-1)^{\textnormal {comp}(C)}w_2(C)\eta_{(t,t_1)} (G_1\setminus C,x)-\notag\\
 &~~~~~w_1(1)\sum_{C\in \Gamma_2 (c)}(-1)^{\textnormal {comp}(C)}w_2(C)\eta_{(w,w_1)} (G_2\setminus C,x).\notag
 \end{align}
Thus
\begin{align}
\sum_{\substack {C\in \Gamma (c)}}(-1)^{\textnormal {comp}(C)}&w_2(C)\eta_{(w,w_1)} (G\setminus C,x)\notag=\\
 &\sum_{\substack {C\in \Gamma (c)}}(-1)^{\textnormal {comp}(C)}w_2(C)\eta_{(t,t_1)} (G_1\setminus C,x)-\notag\\
 &~~~~~w_1(1)\sum_{C\in \Gamma_2 (c)}(-1)^{\textnormal {comp}(C)}w_2(C)\eta_{(w,w_1)} (G_2\setminus C,x),\notag
 \end{align}
and $\phi_{(w,w_1)}(G,x)=\eta_{(w,w_1)} (G,x)+\sum_{C\in \Gamma (c)}(-1)^{\textnormal {comp}(C)}w_2(C)\eta_{(w,w_1)} (G\setminus C,x)$.
\end{proof}

\begin{example}\label{example_5} Let $G$ and $(w,w_1)$ be as in Example \ref{example_2}. Note that the only element in $\Gamma (c)$ is $G$. Now $\eta_{(w,w_1)} (\varnothing,x)=1$, $w_2(G)=b_{v_1v_2}b_{v_2v_3}b_{v_3v_1}+b_{v_1v_3}b_{v_3v_2}b_{v_2v_1}=(1+2i)(2-7i)(-3-2i)+\left(-3+2i)(2+7i\right)(1-2i)=-108$. By Example \ref{example_2}, $\eta_{(w,w_1)} (G,x)=x^3-6x^2-60x+88$. So by Theorem \ref{main_1},
$\phi_{(w,w_1)}(G,x)=x^3-6x^2-60x+88+(-1)(-108)=x^3-6x^2-60x+196$ (see also Example \ref{example_4}).\qed
\end{example}

The following corollary follows from Theorem \ref{main_1}.

\begin{cor}\label{tree_matching_characteristic_identical} If $G$ is a disjoint union of trees (forest) then $\phi_{(w,w_1)}  (G,x)=\eta_{(w,w_1)}  ( G,x)$.\qed
\end{cor}

Note that the converse of Corollary \ref{tree_matching_characteristic_identical} is not true in general (see Example \ref{example_6}). However if the edge weight function is positive real-valued then it is true (Corollary \ref{tree_matching_characteristic_identical_converse}).

\begin{example}\label{example_6} Let $G$ be the graph in Figure 5, $V(G)=\{u_1,u_2,u_3,u_4,u_5\}$ and $(w,w_1)$ be as stated. Here we assume $u_1\equiv 1$, $u_2\equiv 2$, $u_3\equiv 3$, $u_4\equiv 4$ and $u_5\equiv 5$. Note that
\begin{equation}
\phi_{(w,w_1)} (G,x)=\det \begin{pmatrix}
x-2 & -1& 1-i & 0& 0\\
-1 & x-3& -1 & 0& 0\\
1+i & -1& x-4 & -1& -1\\
0 & 0& -1 & x-2& -1\\
0 & 0& -1 & -1& x-3
\end{pmatrix},\notag
\end{equation}
that is $\phi_{(w,w_1)} (G,x)=x^5-14x^4+70x^3-152x^2+135x-35$. Now by using the recurrence in Theorem \ref{basic_recurrence}, $\eta_{(w,w_1)}  ( G,x)=x^5-14x^4+70x^3-152x^2+135x-35$. Therefore $\phi_{(w,w_1)}  (G,x)=\eta_{(w,w_1)}  ( G,x)$ but $G$ is not a forest.

\begin{center}
\begin{pspicture}(0,0)(6,3)
\cnodeput(0.5, 2){1}{}
\cnodeput(0.5, 1){2}{}
\cnodeput(3, 1.5){3}{}
\cnodeput(4.5, 2){4}{}
\cnodeput(4.5, 1){5}{}
\ncline{1}{2}
\ncline{1}{3}
\ncline{3}{4}
\ncline{3}{5}
\ncline{3}{2}
\ncline{4}{5}
\rput(2.5,0){Figure 5.}
\rput(-0.5,1.5){$G=$}
\rput(0.3,1.5){$1$}
\rput(1.5,2){$-1+i$}
\rput(3.7,1.){$1$}
\rput(3.7,2){$1$}
\rput(1.5,1){$1$}
\rput(4.7,1.5){$1$}
\rput(3,1.8){$4$}
\rput(0.5,2.3){$2$}
\rput(0.5,0.7){$3$}
\rput(4.5,2.3){$2$}
\rput(4.5,0.7){$3$}
\rput(0.15,2){$u_1$}
\rput(0.15,1){$u_2$}
\rput(3,1.2){$u_3$}
\rput(4.85,2){$u_4$}
\rput(4.85,1){$u_5$}
\end{pspicture}
\end{center}\qed
\end{example}

\begin{cor}\label{tree_matching_characteristic_identical_converse} Suppose the edge weight function $w$ of $G$ is positive real-valued. Then $G$ is a disjoint union of trees (forest) if and only if $\phi_{(w,w_1)}  (G,x)=\eta_{(w,w_1)}  ( G,x)$.
\end{cor}

\begin{proof} By Corollary \ref{tree_matching_characteristic_identical}, it is sufficient to prove that if $\phi_{(w,w_1)} (G,x)=\eta_{(w,w_1)}  ( G,x)$ then $G$ is a forest.

Suppose $G$ is not a forest. By Theorem \ref{main_1},
\begin{equation}
\phi_{(w,w_1)}(G,x)=\eta_{(w,w_1)} (G,x)+\sum_{C\in \Gamma (c)}(-1)^{\textnormal {comp}(C)}w_2(C)\eta_{(w,w_1)} (G\setminus C,x).\notag
\end{equation}
Therefore $\sum_{C\in \Gamma (c)}(-1)^{\textnormal {comp}(C)}w_2(C)\eta_{(w,w_1)} ( G\setminus C,x)=0$. Let $C$ be the cycle of the least length in $G$. Suppose there are exactly $m$ cycles of such length. Let it be denoted by $C_1,\dots, C_m$. Let us look at the coefficient of $x^{n-\vert C_1\vert}$. Now the summation over all $C_i$, $i=1,\dots,m$, contribute to the coefficient of $x^{n-\vert C_1\vert}$. If $C'\in \Gamma (c)$ and $C'\neq C_i$ for all $i=1,\dots, m$, then it does not contribute to $x^{n-\vert C_1\vert}$ because its length is greater and the degree of $\eta_{(w,w_1)}(G\setminus C',x)$ will be less than $x^{n-\vert C_1\vert}$ (Lemma \ref{degree_weighted}). Each of the $C_i$ contributes exactly $-w_2(C_i)\neq 0$ (by Lemma \ref{real_part_weight_function}, $w_2(C_i)>0$). Therefore the coefficient of $x^{n-\vert C_1\vert}$ is $-\sum_{i=1}^m w_2(C_i)\neq 0$ and  $\sum_{C\in \Gamma (c)}(-1)^{\textnormal {comp}(C)}w_2(C)\eta_{(w,w_1)} ( G\setminus C,x)\neq 0$, a contradiction. Hence $G$ is a forest.
\end{proof}

Note that Theorem \ref{main_1} and Corollary \ref{tree_matching_characteristic_identical_converse} are generalizations of Theorem 4 of \cite{GG} and Corollary 4.2 of \cite{GG}, respectively. This can be seen by taking $w(e)=1$ for all $e\in E(G)$ and $w_1(u)=0$ for all $u\in V(G)$, and noting that $\phi_{(w,w_1)}(G,x)$ is the usual characteristic polynomial of the adjacency matrix of $G$ (also together with Lemma \ref{edge_weighted_to_matching} and Lemma \ref{weighted_to_edge_weighted_matching}).

Now let us discuss the `ordering' in $V(G)$. Before we move on to the next two corollaries, it is a good idea to look at Example \ref{example_ordering} again. Now Corollary \ref{ordering_tree_characteristic} follows from Corollary \ref{tree_matching_characteristic_identical} and Lemma \ref{isomorphi_weighted}.

\begin{cor}\label{ordering_tree_characteristic} Let $G_1$ and $G_2$ be forests with weight function $(w,w_1)$ and $(w',w_1')$, respectively. If $G_1$ is weight-isomorphic to $G_2$, then $\phi_{(w,w_1)} (G_1,x)=\phi_{(w',w_1')} (G_2,x)$.\qed
\end{cor}

\begin{cor}\label{ordering_edge_weight_real_characteristic} Let $G_1$ and $G_2$ be graphs with weight function $(w,w_1)$ and $(w',w_1')$, respectively. If $G_1$ is weight-isomorphic to $G_2$ and the edge weight functions $w, w'$ take non-zero real number, then $\phi_{(w,w_1)} (G_1,x)=\phi_{(w',w_1')} (G_2,x)$.
\end{cor}

\begin{proof} If $G_1$ is a forests then $G_2$ is also a forests. So we are done by Corollary \ref{ordering_tree_characteristic}. Suppose $G_1$ is not a forests. Then $G_2$ is also not a forests. Furthermore every cycle in $G_1$ is also a cycle in $G_2$. Now let us look at the value $w_2(C)$.

Suppose $C=v_1v_2\dots v_mv_1$, $m\geq 3$, is a cycle in $G_1$. Then
\begin{align}
w_2(C) &=b_{v_1v_2}b_{v_2v_3}\dots b_{v_{m-1}v_m}b_{v_mv_1}+b_{v_1v_m}b_{v_mv_{m-1}}\dots b_{v_{3}v_2}b_{v_2v_1}\notag\\
&=2b_{v_1v_2}b_{v_2v_3}\dots b_{v_{m-1}v_m}b_{v_mv_1}\notag\\
&=2w(e_{v_1v_2})w(e_{v_2v_3})\dots w(e_{v_{m-1}v_m})w(e_{v_mv_1})\notag\\
&=2w(C),\notag
\end{align}
where the second and third equalities follow from the fact that $b_{v_{j}v_{j+1}}=\overline{b_{v_{j+1}v_{j}}}=w(e_{v_{j}v_{j+1}})$ (for the edge weight function $w$ take non-zero real number).

Suppose $C$ is a disjoint union of $k$ cycles $C_1,\dots,C_k$. Then $w_2(C)=w_2(C_1)\ \dots\ w_2(C_k)=2^kw(C_1)\dots w(C_k)=2^kw(C)$. So the value of $w_2(C)$ is equal to $2^{\textnormal{comp}(C)}$ times
the product of all the weights on the edges in $C$.

Similarly $w_2'(C)$ is equal to $2^{\textnormal{comp}(C)}$ times
the product of all the weights on the edges in $C$. Therefore $w_2'(C)=w_2(C)$. It then follows from Theorem \ref{main_1} and Lemma \ref{isomorphi_weighted} that $\phi_{(w,w_1)} (G_1,x)=\phi_{(w',w_1')} (G_2,x)$.
\end{proof}

The next corollary follows from Corollary \ref{tree_matching_characteristic_identical} and the fact that all eigenvalues of a Hermitian matrix are real (see \cite[Theorem 7.5.1 on p. 209]{Mir}).
\begin{cor}\label{root_tree} If $T$ is a tree  then the roots of $\eta_{(w,w_1)}(T,x)$ are real. \qed
\end{cor}

Now if $w_1(u)=0$ for all $u\in V(T)$, we can say further on where it's roots lie. This will done in the next corollary.

\begin{cor}\label{root_tree_extra} Let $T$ be a tree. Suppose $w_1(u)=0$ for all $u\in V(T)$. If the maximum valency $\Delta$ of $T$ is greater than 1, then the roots of $\eta_{(w,w_1)}(T,x)$  lie in the interval $[-2b_0\sqrt {\Delta-1},2b_0\sqrt {\Delta-1}]$, where $b_0=\max_{e\in E(T)} \vert w(e)\vert$.
\end{cor}

\begin{proof} First note that $\phi_{(w,w_1)}(T,x)=\eta_{(w,w_1)}(T,x)=\mu_{w}(T,x)$ (Corollary \ref{tree_matching_characteristic_identical} and Lemma \ref{weighted_to_edge_weighted_matching}). Let $B=B_{(w,w_1)}(T)=[b_{uv}]$. Then $b_{uu}=0$ for all $u$. Let $b_0=\max_{u,v} \vert b_{uv}\vert$. Then $b_0=\max_{e\in E(T)} \vert w(e)\vert$. Let $C=[c_{uv}]$ where $c_{uv}=0$ if $b_{uv}=0$ and $c_{uv}=b_0$ if $b_{uv}\neq 0$. Set $w_0(e)=1$ for all $e\in E(T)$. Then $C=b_0B_{(w_0,w_1)}(T)$. Let $B_{(w_0,w_1)}(T)=[d_{uv}]$. Note that  $B_{(w_0,w_1)}(T)$ is the adjacency matrix of $T$.

Now let $\lambda$ be an eigenvalue of $B$ and $\mathbf x_0=(x_1,\dots ,x_n)$ be its corresponding eigenvector. Then
\begin{equation}
\vert \lambda\vert=\left \vert \min_{1\leq u\leq n, x_u\neq 0} \frac {\sum_{k=1}^n b_{uk}x_k}{x_u}\right\vert\leq b_0\min_{1\leq u\leq n, x_u\neq 0}  \frac {\sum_{k=1}^n d_{uk}\vert x_k\vert}{\vert x_u\vert}\leq b_0r,\notag
\end{equation}
where $r$ is a positive eigenvalue of $B_{(w_0,w_1)}(T)$ for which the absolute value of any eigenvalues of $B_{(w_0,w_1)}(T)$ is at most $r$ (see the discussion on p. 534, Proposition 2 on p. 535 and Theorem 1 on p. 536 of \cite{LT}).
By Theorem 6.3 on p. 87 of \cite{G0}, $r\leq 2\sqrt{\Delta-1}$. Hence $\vert \lambda\vert\leq 2b_0\sqrt{\Delta-1}$.
\end{proof}

\section{The Path-tree}

The notion of a path-tree of a graph was first introduced by  Godsil \cite[Section 2]{G1} (see also \cite[Section 6.1]{G0}). Let $G$ be a graph with a vertex $u$. The {\em path-tree} $T(G,u)$ is the tree with the paths in $G$ starting at $u$ as its vertices, and two such paths are joined by an edge if one is a maximal subpath of the other. The vertex $u$ is itself a path, and so it is a vertex of $T(G,u)$ and will also be denoted by $u$. Now let us assign the weight to $T(G,u)$. Note that two vertices, $p_1$ and $p_2$ in $V(T(G,u))$ are joined by an edge if and only if $p_1=p_2uv$ or $p_2=p_1uv$ for some edge $e_{uv}\in E(G)$. We set $w^{T}(e_{p_1p_2})=w(e_{uv})$.

Let $p$ be a vertex in $T(G,u)$. If $p=u$, we set $w_{1}^{T}(p)=w_1(u)$. If $p$ is a path with length at least 1, then we set $w_{1}^{T}(p)=w_1(v)$, where $v\neq u$ is an endpoint of $p$.

So for each graph $G$ and weight function $(w,w_1)$, there corresponds a path-tree $T(G,u)$ and weight function $(w^{T},w_{1}^{T})$. For convenience, when there is no confusion, we shall write $w^T$ as $w$ and $w_{1}^{T}$ as $w_1$.

Note that if $G$ is a tree, then $G$ is weight-isomorphic to $T(G,u)$. This can be seen by part (b) of Lemma 2.4 of \cite{G1} and keeping track of the weights on the edges and vertices.

The next theorem is a generalization of Theorem 2.5 of \cite{G1}. However  it's proof is similar to that in \cite{G1}. In fact it can be proved by using Lemma 2.4 of \cite{G1} and Theorem \ref{basic_recurrence}. The details of the proof are omitted.

\begin{thm}\label{main_2}
Let $u$ be a vertex in $G$ and $T:=T(G,u)$ be the path-tree of $G$ with respect to $u$. Then
\[ \frac{\eta_{(w,w_1)}(G \setminus u,x)}{\eta_{(w,w_1)}(G,x)} = \frac{\eta_{(w,w_1)}(T \setminus u,x)}{\eta_{(w,w_1)}(T,x)}. \]\qed
\end{thm}

The next corollary follows easily from Theorem \ref{main_2}. For the sake of completeness, we shall give a proof.

\begin{cor}\label{divisibility_matching}  Let $G$ be a connected graph and $u$ be a vertex in $G$. Let $T=T( G,u)$ be the path tree of $G$. Then $\eta_{(w,w_1)} (G,x)$ divides $\eta_{(w,w_1)}(T,x)$.
\end{cor}

\begin{proof} If $G$ is a tree then by part (b) of Lemma 2.4 of \cite{G1}, we deduce that $G$ is weight-isomorphic to $T$. It then follows from Lemma \ref{isomorphi_weighted} that $\eta_{(w,w_1)} (G,x)=\eta_{(w,w_1)}(T,x)$. Hence  the corollary holds. We may assume inductively that the corollary holds for all connected subgraphs of $G$. Let $G\setminus u= H_1\cup\cdots\cup H_k$ where $H_1,\dots,H_k$ are components of $G\setminus u$.  Then by part (a) of Theorem \ref{basic_recurrence},
\begin {equation}
\eta_{(w,w_1)} (G\setminus u,x)=\prod_{j=1}^k\eta_{(w,w_1)} (H_j,x).\notag
\end {equation}

For each $j$, let $v_j\in V(H_j)$ be such that $e_{uv_j}$ is an edge in $G$. By part (c) of Lemma 2.4 of \cite{G1}, we deduce that (keeping track of the weights) the component of $T(G,u)\setminus u$ that contains the vertex $p_0=uv_j$ is isomorphic to the path tree $T(H_j,v_j)$. Note also that $T( G\setminus u,v_j)=T(H_j,v_j)$. Therefore by part (a) of Theorem \ref{basic_recurrence}, we deduce that $\prod_{j=1}^k\eta_{(w,w_1)} (T( G\setminus u,v_j),x)$ divides $\eta_{(w,w_1)} (T(G,u)\setminus u,x)$.
 By induction hypothesis, $\eta_{(w,w_1)} (H_j,x)$ divides $\eta_{(w,w_1)}(T( H_j,v_j),x)$.  Therefore $\eta_{(w,w_1)} (G\setminus u,x)$ divides $\eta_{(w,w_1)} (T(G,u)\setminus u,x)$.  By Theorem \ref{main_2},
\begin {equation}
\frac {\eta_{(w,w_1)}(G\setminus u,x)}{\eta_{(w,w_1)}(G,x)}=\frac {\eta_{(w,w_1)}(T\setminus u,x)}{\eta_{(w,w_1)}(T,x)}.\notag
\end {equation}
Hence $\eta_{(w,w_1)} (G,x)$ divides $\eta_{(w,w_1)}(T,x)$.
\end{proof}

The following two corollaries follows from part (a) of Theorem \ref{basic_recurrence}, Corollary \ref{divisibility_matching}, Corollary \ref{root_tree} and Corollary \ref{root_tree_extra}.

\begin{cor}\label{root_graph}  Let $G$ be a graph.  Then the roots of $\eta_{(w,w_1)} (G,x)$ are real.\qed
\end{cor}

\begin{cor}\label{root_graph_extra}  Let $G$ be a graph. Suppose $w_1(u)=0$ for all $u\in V(G)$. If the maximum valency $\Delta$ of $G$ is greater than 1, then the roots of $\eta_{(w,w_1)}(G,x)$ lie in the interval $[-2b_0\sqrt {\Delta-1},2b_0\sqrt {\Delta-1}]$, where $b_0=\max_{e\in E(G)} \vert w(e)\vert$.\qed
\end{cor}

\section{Vertex classification}

The following lemma can be deduced using equation (2) on p. 29 of \cite{G0} (see the proof of Theorem 5.3 on p. 29 of \cite{G0} for the details).

\begin{lm}\label{interlacing_matrix} Let $B=[b_{uv}]$ be an $n\times n$ Hermitian matrix. Let $\lambda_1,\lambda_2,\dots, \lambda_n$ be all the eigenvalues of $B$ with $\lambda_1\geq\lambda_2\geq \cdots\geq \lambda_n$.  Let $\theta_1,\theta_2,\dots, \theta_{n-1}$ be all the eigenvalues of $B(u;u)$ with $\theta_1\geq\theta_2\geq \cdots\geq \theta_{n-1}$ \textnormal {($B(u;u)$ is the matrix obtained from $B$ by deleting the $u$ row and the $u$ column)}.  Then
\begin {equation}
\lambda_1\geq \theta_1\geq \lambda_2\geq\theta_2\geq\cdots\geq \lambda_{n-1}\geq\theta_{n-1}\geq\lambda_n.\notag
\end {equation}\hfill\qed
\end{lm}

\begin{lm}\label{interlacing_root} Let  $u\in V(G)$ and $\theta$ be a real number. Then
\begin {equation}
\textnormal {mult} (\theta,G,\eta_{(w,w_1)})-1\leq \textnormal {mult} (\theta, G\setminus u, \eta_{(w,w_1)})\leq \textnormal {mult} (\theta, G,\eta_{(w,w_1)})+1.\notag
\end {equation}
\end{lm}

\begin{proof} By Theorem \ref{main_2},
\begin {equation}
\frac{\eta_{(w,w_1)}(G \setminus u,x)}{\eta_{(w,w_1)}(G,x)} = \frac{\eta_{(w,w_1)}(T \setminus u,x)}{\eta_{(w,w_1)}(T,x)}.\notag
\end {equation}
Let $\eta_{(w,w_1)}( G\setminus u,x)/\eta_{(w,w_1)} (G,x)=(x-\theta)^rh(x)/g(x)$ where $h(x)$ and $g(x)$ are polynomials such that $h(\theta)\neq 0\neq g(\theta)$ and $r=\textnormal {mult} (\theta, G\setminus u,\eta_{(w,w_1)})-\textnormal {mult} (\theta, G,\eta_{(w,w_1)})$.

By Corollary \ref{tree_matching_characteristic_identical},  $\phi_{(w,w_1)} (T,x)=\eta_{(w,w_1)} (T,x)$  and $\phi_{(w,w_1)} (T\setminus u,x)=\eta_{(w,w_1)} (T\setminus u,x)$. Let $\textnormal {mult} (\theta, T,\eta_{(w,w_1)})=m$. By Lemma \ref{interlacing_matrix}, we deduce that $m-1\leq \textnormal {mult} (\theta, T\setminus u,\eta_{(w,w_1)})\leq m+1$. Now $\eta_{(w,w_1)}(T\setminus u,x)/\eta_{(w,w_1)}(T,x)=(x-\theta)^rh(x)/g(x)$ with $r=\textnormal {mult} (\theta, T\setminus u,\eta_{(w,w_1)})-\textnormal {mult} (\theta, T,\eta_{(w,w_1)})$. So, $-1\leq r\leq 1$ and the lemma holds.
\end{proof}

The following definition is motivated by Lemma \ref{interlacing_root} and followed Godsil's approach \cite[Section 3]{G2}

\begin {dfn}\label{vertex_classification_dfn} For any $u\in V(  G)$,
\begin {itemize}
\item [(a)] $u$ is $(\theta,w,w_1)$-\emph {essential} if $\textnormal {mult} (\theta,   G\setminus u,\eta_{(w,w_1)})=\textnormal {mult} (\theta,   G,\eta_{(w,w_1)})-1$,
\item [(b)] $u$ is $(\theta,w,w_1)$-\emph {neutral} if $\textnormal {mult} (\theta,   G\setminus u,\eta_{(w,w_1)})=\textnormal {mult} (\theta,   G,\eta_{(w,w_1)})$,
\item [(c)] $u$ is $(\theta,w,w_1)$-\emph {positive} if $\textnormal {mult} (\theta,   G\setminus u,\eta_{(w,w_1)})=\textnormal {mult} (\theta,   G,\eta_{(w,w_1)})+1$.
\end {itemize}
Furthermore if $u$ is not $(\theta,w,w_1)$-essential but it is adjacent to some $(\theta,w,w_1)$-essential vertex, we say $u$ is $(\theta,w,w_1)$-\emph{special}. A graph $G$ is said to be $(\theta,w,w_1)$-\emph{critical} if all vertices in $G$ are $(\theta,w,w_1)$-essential and $\textnormal {mult} (\theta, G,\eta_{(w,w_1)})=1$.
\end {dfn}

The following lemma is a generalization of \cite[Lemma 3.1]{G2}. However, its proof is similar to that in \cite{G2}. In fact we just need to compare the multiplicity of $\theta$ as a root on both sides on the equation in part (d) of Theorem \ref{basic_recurrence}. The details are omitted.

\begin{lm}\label{essential_vertex_exists} For any graph $G$, it has at least one $(\theta,w,w_1)$-essential vertex provided that $\theta$ is a root of $\eta_{(w,w_1)}(G,x)$.\qed
\end{lm}

The next lemma is a generalization of the Heilmann-Lieb equation \cite[Lemma 2.4]{G2} (see also \cite[Theorem 6.3]{HL} and \cite[Lemma 4.1 on p. 104]{G0}). This can be seen by taking $w(e)=1$ for all $e\in E(G)$ and $w_1(u)=0$ for all $u\in V(G)$ (also together with Lemma \ref{edge_weighted_to_matching} and Lemma \ref{weighted_to_edge_weighted_matching}). It can be proved by using induction on the number of edges and Theorem \ref{basic_recurrence} (following a similar argument as in \cite{G0}).

\begin{lm}\label{heilman_leib} Let $u,v \in V(G)$ and $u\neq v$. Then
\begin {align}
\eta_{(w,w_1)}( G\setminus u,x)\eta_{(w,w_1)}&(G\setminus v,x)-\eta_{(w,w_1)}(G,x)\eta_{(w,w_1)}(G\setminus uv,x)=\notag\\
&\sum_{p\in P_{uv}(G)} \left (\vert w(p)\vert\eta_{(w,w_1)}( G\setminus p,x)\right)^2,\tag{*}
\end {align}
where $P_{uv}(G)$ is the set of all the paths in $G$ that have $u$ and $v$ as endpoints.\qed
\end{lm}

The next corollary is a generalization of \cite[Corollary 2.5]{G2}. It is a consequence of Lemma \ref{heilman_leib} and Lemma \ref{interlacing_root} (following a similar argument as in \cite{G2}).  The details are omitted.

\begin{cor}\label{path_interlacing} Let $p$ be a path of length at least 1 in $G$. Then
\begin {equation}
\textnormal {mult} (\theta,G\setminus p,\eta_{(w,w_1)})\geq \textnormal {mult} (\theta, G,\eta_{(w,w_1)})-1.\notag
\end {equation}\qed
\end{cor}

\begin{cor}\label{path_cover_maximum_multiplicity}  Let $G$ be a graph. Then
\begin{itemize}
\item [\textnormal{(a)}] the maximum multiplicity of a root of $\eta_{(w,w_1)}(G,x)$ is at most equal to the number of vertex-disjoint paths required to cover $G$,
\item [\textnormal{(b)}] the number of distinct roots of  $\eta_{(w,w_1)}(G,x)$ is at least equal to the number of vertices in the longest path in $G$.
\end{itemize}
\end{cor}

\begin{proof} (a)  Let $p_1,\dots, p_k$ be all the vertex-disjoint paths that  cover $G$. By Corollary \ref{path_interlacing} and Lemma \ref{interlacing_root} (in the case if $p_j$ is a single vertex),  we have $\textnormal {mult} (\theta, G\setminus p_1,\eta_{(w,w_1)})\geq \textnormal {mult} (\theta,G,\eta_{(w,w_1)})-1$ and inductively $\left.\textnormal {mult} (\theta, G\setminus (p_1\cup\cdots\cup p_{k-1}\cup p_k),\eta_{(w,w_1)})\right.\geq \textnormal {mult} (\theta,G,\eta_{(w,w_1)})-k$. Note that $\left.\textnormal {mult} (\theta, G\setminus (p_1\cup\cdots\cup p_{k-1}\cup p_k),\eta_{(w,w_1)})\right.=0$. Hence $\textnormal {mult} (\theta,G,\eta_{(w,w_1)})\leq k$.

\noindent
(b) Let $p$ be a path in $G$. Let $\Theta$ be the set of all distinct roots of $\eta_{(w,w_1)}(G,x)$. Now $\textnormal {mult} (\theta, G\setminus p,\eta_{(w,w_1)})\geq \textnormal {mult} (\theta,G,\eta_{(w,w_1)})-1$ (Corollary \ref{path_interlacing} and Lemma \ref{interlacing_root})  implies that
\begin{equation}
\sum_{\theta\in\Theta} \textnormal {mult} (\theta,  G\setminus p,\eta_{(w,w_1)})\geq \sum_{\theta\in\Theta}\textnormal {mult} (\theta,G,\eta_{(w,w_1)})-\vert\Theta \vert.\notag
\end{equation}
 Since $\sum_{\theta\in\Theta} \textnormal {mult} (\theta, G\setminus p,\eta_{(w,w_1)})=V(G\setminus p)$ and $\sum_{\theta\in\Theta} \textnormal {mult} (\theta, G,\eta_{(w,w_1)})=V(G)$, we have $\vert \Theta\vert\geq  \vert V(p) \vert$.
\end{proof}

\begin{dfn}\label{dfn_essential_path} A path $p$ in $G$ is said to be $(\theta,w,w_1)$-\emph{essential} if
\begin{equation}
\textnormal {mult} (\theta,G\setminus p,\eta_{(w,w_1)})=\textnormal {mult} (\theta, G,\eta_{(w,w_1)})-1.\notag
\end{equation}
 So if a path $q$ is not $(\theta,w,w_1)$-essential, then  $\textnormal {mult} (\theta,G\setminus q,\eta_{(w,w_1)})\geq \textnormal {mult} (\theta, G,\eta_{(w,w_1)})$ (Corollary \ref{path_interlacing}).
\end{dfn}

Part (a) and (b) of the next lemma are generalizations of \cite[Lemma 3.2]{G2} and \cite[Lemma 3.3]{G2}, respectively. However their proofs are similar to that in \cite{G2}. In fact for part (a), it can be deduced from part (c) of Theorem \ref{basic_recurrence} and Corollary \ref{path_interlacing}, whereas for part (b), it can be deduced from Lemma \ref{heilman_leib} and Lemma \ref{interlacing_root}. The details are omitted.

\begin{lm}\label{general_properties_path} Let $G$ be a graph and $\theta$ be a root of $\eta_{(w,w_1)}(G,x)$. Then
\begin {itemize}
\item [\textnormal{(a)}]  for any $(\theta,w,w_1)$-essential vertex $u$ with $\theta\neq w_1(u)$,  there is a vertex $v$ such that the path $p=uv$ is $(\theta,w,w_1)$-essential,
\item [\textnormal{(b)}] if $u$ is not $(\theta,w,w_1)$-essential, then for any path $p$ that ends with $u$, $p$ is not $(\theta,w,w_1)$-essential.\qed
\end {itemize}
\end{lm}

\section{Gallai-Edmonds decomposition}

We shall begin by showing that a $(\theta,w,w_1)$-special vertex is $(\theta,w,w_1)$-positive. This is a generalization of Corollary 4.3 of \cite{G2}.

\begin{lm}\label{special_vertex_positive} Let $u\in V(G)$. If $u$ is $(\theta,w,w_1)$-special then $u$ is $(\theta,w,w_1)$-positive.
\end{lm}

\begin{proof} By Definition \ref{vertex_classification_dfn}, there is a $(\theta,w,w_1)$-essential vertex $v$ such that $e_{uv}\in E(G)$. Since $u$ is not $(\theta,w,w_1)$-essential, by part (b) of Lemma \ref{general_properties_path}, the path $p=uv$ is not $(\theta,w,w_1)$-essential.  By Corollary \ref{path_interlacing} (see also Definition \ref{dfn_essential_path}), $\textnormal {mult} (\theta,G\setminus uv,\eta_{(w,w_1)})=\textnormal {mult} (\theta, G\setminus p,\eta_{(w,w_1)})\geq k$ where $k=\textnormal {mult} (\theta,G,\eta_{(w,w_1)})$. Also $u$ is either $(\theta,w,w_1)$-positive or $(\theta,w,w_1)$-neutral (Lemma \ref{interlacing_root}).

Suppose $u$ is $(\theta,w,w_1)$-neutral. Then $\textnormal {mult} (\theta, G\setminus u,\eta_{(w,w_1)})=k$. Now the multiplicity of $\theta$ as a root of $\eta_{(w,w_1)}(G\setminus u,x)\eta_{(w,w_1)}(G\setminus v,x)$ is exactly $2k-1$ and the multiplicity of $\theta$ as a root of $\left.\eta_{(w,w_1)}(G,x)\eta_{(w,w_1)}(G\setminus uv,x)\right.$ is at least $2k$. This implies that the multiplicity of $\theta$ as a root of $\sum_{p\in P_{uv}(G)} (\vert w(p)\vert\eta_{(w,w_1)}(G\setminus p,x))^2$ is $2k-1$ (Lemma \ref{heilman_leib}) which is a contradiction since the multiplicity of $\theta$ as a root of $\sum_{p\in P_{uv}(G)} (\vert w(p)\vert\eta_{(w,w_1)}(G\setminus p,x))^2$ is at least $2k$. Hence $u$ is $(\theta,w,w_1)$-positive.
\end{proof}

By Definition \ref{vertex_classification_dfn} and Lemma \ref{special_vertex_positive}, we have
\begin {equation}
V(G)=D_{(\theta,w,w_1)}(G)\cup A_{(\theta,w,w_1)}(G)\cup P_{(\theta,w,w_1)}(G)\cup N_{(\theta,w,w_1)}(G),\notag
\end {equation}
where
\begin {itemize}
\item [] $D_{(\theta,w,w_1)}(G)$ is the set of all $(\theta,w,w_1)$-essential vertices in $G$,
\item [] $A_{(\theta,w,w_1)}(G)$ is the set of all $(\theta,w,w_1)$-special vertices in $G$,
\item [] $N_{(\theta,w,w_1)}(G)$ is the set of all $(\theta,w,w_1)$-neutral vertices in $G$,
\item [] $P_{(\theta,w,w_1)}(G)=Q_{(\theta,w,w_1)}(G)\setminus A_{(\theta,w,w_1)}(G)$, where  $Q_{(\theta,w,w_1)}(G)$ is the set of all $(\theta,w,w_1)$-positive vertices in $G$,
\end {itemize}
is a partition of $V(G)$.

The following lemma is a generalization of Theorem 4.2 of \cite{G2} and it's proof is similar to that in \cite{G2}. In fact for part (a), it can be deduced from Lemma \ref{interlacing_root}, whereas for part (b) and (c), it can be deduced by comparing the multiplicity of $\theta$ as a root in the equation (*) of  Lemma \ref{heilman_leib}. The details are omitted.

\begin {lm}\label{positive_vertex_deletion} Let $\textnormal {mult} (\theta,G,\eta_{(w,w_1)})=k$ and $v\in V(G)$ be $(\theta,w,w_1)$-positive. Then
\begin {itemize}
\item [\textnormal{(a)}] if $u$ is $(\theta,w,w_1)$-essential in $G$ then it is $(\theta,w,w_1)$-essential in $G\setminus v$,
\item [\textnormal{(b)}] if $u$ is $(\theta,w,w_1)$-positive in $G$ then it is $(\theta,w,w_1)$-essential or $(\theta,w,w_1)$-positive in $G\setminus v$,
\item [\textnormal{(c)}] if $u$ is $(\theta,w,w_1)$-neutral in $G$ then it is $(\theta,w,w_1)$-essential or $(\theta,w,w_1)$-neutral in $G\setminus v$.\qed
\end {itemize}
\end {lm}

The following lemma can be proved similarly by comparing the multiplicity of $\theta$ on both sides of (*) of Lemma \ref{heilman_leib} (see \cite[Proposition 2.9]{KC}).

\begin{lm}\label{neutral_vertex_deletion} Let $\textnormal {mult} (\theta,G,\eta_{(w,w_1)})=k$ and $v\in V(G)$ be $(\theta,w,w_1)$-neutral. Then
\begin {itemize}
\item [\textnormal{(a)}] if $u$ is $(\theta,w,w_1)$-essential in $G$ then it is $(\theta,w,w_1)$-essential in $G\setminus v$,
\item [\textnormal{(b)}] if $u$ is $(\theta,w,w_1)$-positive in $G$ then it is either $(\theta,w,w_1)$-positive or $(\theta,w,w_1)$-neutral  in $G\setminus v$,
\item [\textnormal{(c)}] if $u$ is $(\theta,w,w_1)$-neutral in $G$ then it is either $(\theta,w,w_1)$-positive or $(\theta,w,w_1)$-neutral  in $G\setminus v$.\qed
\end {itemize}
\end{lm}

\begin{lm}\label{essential_vertex_path}  Let $v,z$ be $(\theta,w,w_1)$-essential in $G$ and $\textnormal {mult} (\theta, G\setminus vz,\eta_{(w,w_1)})\geq \textnormal {mult} (\theta, G,\eta_{(w,w_1)})-1$. If  $p$ is a path in $G$ with endpoints $v$ and $z$, then $p$ is $(\theta,w,w_1)$-essential in $G$.
\end{lm}

\begin{proof} Note that the multiplicity of $\theta$ as a root of $\eta_{(w,w_1)}( G\setminus z,x)\eta_{(w,w_1)}(G\setminus v,x)$ is $2k-2$, where $k=\textnormal {mult} (\theta, G,\eta_{(w,w_1)})$. Also the multiplicity of $\theta$ as a root of $\eta_{(w,w_1)}(G,x)\eta_{(w,w_1)}(G\setminus vz,x)$ is at least $2k-1$ (for $\textnormal {mult} (\theta, G\setminus vz,\eta_{(w,w_1)})\geq k-1$). This implies that the multiplicity of $\theta$ as a root of $\sum_{q\in P_{vz}(G)} \left (\vert w(q)\vert\eta_{(w,w_1)}( G\setminus q,x)\right)^2$ is $2k-2$ (Lemma \ref{heilman_leib}). Thus $\textnormal {mult} (\theta, G\setminus q,\eta_{(w,w_1)})=k-1$ for all $q \in P_{vz}(G)$; in particular
$\textnormal {mult} (\theta, G\setminus p,\eta_{(w,w_1)})=k-1$, i.e. $p$ is $(\theta,w,w_1)$-essential (Definition \ref{dfn_essential_path}).
\end{proof}

The next lemma is somewhat similar to \cite[Lemma 4.1]{KC}. Basically it is the essence of \cite[Lemma 4.1, Lemma 4.2 and Lemma 4.3]{KC}.

\begin{lm}\label{technical_lemma_1} Let $u,v,z\in V(G)$ be such that $u$ is adjacent to $v$ and $z$. Suppose
 $u$ is $(\theta,w,w_1)$-special and $v$ is $(\theta,w,w_1)$-essential in $G$. Let $G'=G-e_{uz}$. Then
$\textnormal {mult} (\theta, G',\eta_{(w,w_1)})=\textnormal {mult} (\theta, G,\eta_{(w,w_1)})$, $u$ is $(\theta,w,w_1)$-positive in $G'$. Furthermore if the path $p=vuz$ is not $(\theta,w,w_1)$-essential in $G$, then $u$ is $(\theta,w,w_1)$-special in $G'$.
\end{lm}

\begin{proof} Let $\textnormal {mult} (\theta, G,\eta_{(w,w_1)})=k$.  Then  $\textnormal {mult} (\theta, G\setminus u,\eta_{(w,w_1)})=k+1$ and $\textnormal {mult} (\theta, G\setminus v,\eta_{(w,w_1)})=k-1$. It is not hard to deduce from Lemma \ref{interlacing_root}, Lemma \ref{special_vertex_positive} and part (a) of Lemma \ref{positive_vertex_deletion}, that $\textnormal {mult} (\theta, G\setminus uz,\eta_{(w,w_1)})\geq k$, $\textnormal {mult} (\theta, G\setminus uv,\eta_{(w,w_1)})=k$.

Now $\textnormal {mult} (\theta, G'\setminus u,\eta_{(w,w_1)})=\textnormal {mult} (\theta, G\setminus u,\eta_{(w,w_1)})=k+1$ (for $G'\setminus u=G\setminus u$) implies that $\textnormal {mult} (\theta, G',\eta_{(w,w_1)})=k$, $k+1$ or $k+2$ (Lemma \ref{interlacing_root}).

Suppose $\textnormal {mult} (\theta, G',\eta_{(w,w_1)})=k+2$. Then by Corollary \ref{path_interlacing}, $\textnormal {mult} (\theta, G'\setminus uv,\eta_{(w,w_1)})\geq k+1$, a contradiction (for $\textnormal {mult} (\theta, G'\setminus uv,\eta_{(w,w_1)})=\textnormal {mult} (\theta, G\setminus uv,\eta_{(w,w_1)})=k$).

Suppose $\textnormal {mult} (\theta, G',\eta_{(w,w_1)})=k+1$. Then $u$ is  $(\theta,w,w_1)$-neutral in $G'$ (for
$\textnormal {mult} (\theta, G'\setminus u,\eta_{(w,w_1)})=k+1$). By part (b) of Lemma \ref{general_properties_path}, the path $uv$ is not $(\theta,w,w_1)$-essential in $G'$. It then follows from
Corollary \ref{path_interlacing} and Definition \ref{dfn_essential_path}, that $\left.\textnormal {mult} (\theta, G'\setminus uv,\eta_{(w,w_1)})\right.\geq k+1$ , a contradiction (for $\textnormal {mult} (\theta, G'\setminus uv,\eta_{(w,w_1)})=k$).

So $\textnormal {mult} (\theta, G',\eta_{(w,w_1)})=k$ and $u$ is $(\theta,w,w_1)$-positive in $G'$.

Suppose $p=vuz$ is not $(\theta,w,w_1)$-essential in $G$. Then by  Definition \ref{dfn_essential_path}, $\textnormal {mult} (\theta, G\setminus vuz,\eta_{(w,w_1)})\geq k$. Now by part (b) of Theorem \ref{basic_recurrence}, $\eta_{(w,w_1)} (G\setminus v,x)=\eta_{(w,w_1)} (G'\setminus v,x)-\vert w(e_{uz})\vert^2\eta_{(w,w_1)} (G\setminus vuz, x)$. Since  $\textnormal {mult} (\theta, G\setminus v,\eta_{(w,w_1)})=k-1$, we deduce that $\textnormal {mult} (\theta, G'\setminus v,\eta_{(w,w_1)})=k-1$. Hence $v$ is $(\theta,w,w_1)$-essential  and $u$ is $(\theta,w,w_1)$-special in $G'$ (for $u$ is adjacent to $v$ and $u$ is not  $(\theta,w,w_1)$-essential in $G'$).
\end{proof}

The next lemma is a generalization of \cite[Proposition 5.1]{KC} and it can be proved using similar argument as in \cite{KC}. Nevertheless we shall give the details.

\begin{lm}\label{special_deletion_lemma} Let $u$ be $(\theta,w,w_1)$-special in $G$ and $v$ be $(\theta,w,w_1)$-essential in $G\setminus u$. Then $v$ is either $(\theta,w,w_1)$-positive  or $(\theta,w,w_1)$-essential in $G$.
\end{lm}

\begin{proof} Let $\textnormal {mult} (\theta, G,\eta_{(w,w_1)} )=k$. Then $\textnormal {mult} (\theta, G\setminus uv,\eta_{(w,w_1)} )=k$ (Lemma \ref{special_vertex_positive}). Suppose $v$ is $(\theta,w,w_1)$-neutral in $G$. Then $u$ is $(\theta,w,w_1)$-neutral in $G\setminus v$. But $u$ is adjacent to a $(\theta,w,w_1)$-essential vertex $z$ in $G$, so by part (a) of Lemma \ref{neutral_vertex_deletion}, $z$ is $(\theta,w,w_1)$-essential in $G\setminus v$, which means that $u$ is $(\theta,w,w_1)$-special and thus $(\theta,w,w_1)$-positive in $G\setminus v$ (Lemma \ref{special_vertex_positive}) a contradiction. Hence $v$ is either $(\theta,w,w_1)$-positive  or $(\theta,w,w_1)$-essential in $G$.
\end{proof}

A vertex is said to be an \emph{isolated vertex} in $G$ if it is not adjacent to any other vertices in $G$.

\begin{lm}\label{isolated_vertex} Let $u$ be an isolated vertex in $G$. Then
\begin{itemize}
\item [\textnormal {(a)}] if $\theta=w_1(u)$ then $u$ is $(\theta,w,w_1)$-essential in $G$,
\item [\textnormal {(b)}] if $\theta\neq w_1(u)$ then $u$ is $(\theta,w,w_1)$-neutral in $G$.
\end{itemize}
\end{lm}

\begin {proof} The lemma follows by comparing the multiplicity of $\theta$ as a root of both sides of the equation $\eta_{(w,w_1)} (G,x)=$ $\left (x-w_1(u)\right)\eta_{(w,w_1)} (G\setminus u,x)$ (part (c) of Theorem \ref{basic_recurrence}).
\end{proof}

The following fact was first observed by Chen and Ku in their proof of \cite[Theorem 1.5]{KC} for the classification of vertices, using the root of the usual matching polynomial $\mu (G,x)$. We shall give the details of the proof.

\begin{lm}\label{degree_special_vertex} Let $u$ be $(\theta,w,w_1)$-special in $G$. Then the degree of $u$ is at least two.
\end{lm}

\begin {proof}  Suppose the contrary. Then the degree of $u$ is one and the vertex adjacent to $u$, say $z$ is $(\theta,w,w_1)$-essential. Now $\textnormal {mult} (\theta, G\setminus u,\eta_{(w,w_1)})=k+1$ (Lemma \ref{special_vertex_positive}), $\left.\textnormal {mult} (\theta, G\setminus z,\eta_{(w,w_1)})\right.=k-1$, where $k=\textnormal {mult} (\theta, G,\eta_{(w,w_1)})$.
Also by  part (a) of Lemma \ref{positive_vertex_deletion}, $\textnormal {mult} (\theta, G\setminus uz,\eta_{(w,w_1)})=k$.

Let $G'=G-e_{uz}$. From $\eta_{(w,w_1)} (G,x)=\eta_{(w,w_1)} (G',x)-\vert w(e_{uz})\vert^2\eta_{(w,w_1)} (G\setminus uz, x)$ (part (b) of Theorem \ref{basic_recurrence}), we deduce that $\textnormal {mult} (\theta, G',\eta_{(w,w_1)})\geq k$. On the other hand, $\textnormal {mult} (\theta, G'\setminus z,\eta_{(w,w_1)})=\textnormal {mult} (\theta, G\setminus z,\eta_{(w,w_1)})=k-1$ (for $G'\setminus z=G\setminus z$) implies that $\textnormal {mult} (\theta, G',\eta_{(w,w_1)})=k-2$, $k-1$ or $k$ (Lemma \ref{interlacing_root}). Hence $\textnormal {mult} (\theta, G',\eta_{(w,w_1)})=k$.

Now $\textnormal {mult} (\theta, G'\setminus u,\eta_{(w,w_1)})=\textnormal {mult} (\theta, G\setminus u,\eta_{(w,w_1)})=k+1$, that is $u$ is $(\theta,w,w_1)$-positive  in $G'$. But this contradicts Lemma \ref{isolated_vertex} (for $u$ is an isolated vertex in $G'$). Hence the degree of $u$ is at least two.
\end{proof}

\begin{lm}\label{union_essential_vertex} Let $G$ be the union of two graphs $G_1$ and $G_2$. Let $u\in G_1$ and $v\in G_2$. Then $v$ is $(\theta,w,w_1)$-essential in $G$ if and only if it is $(\theta,w,w_1)$-essential in $G\setminus u$.
\end{lm}

\begin{proof} By part (a) of Theorem \ref{basic_recurrence}, we deduce that
\begin{align}
\textnormal {mult} (\theta, G,\eta_{(w,w_1)})&=\textnormal {mult} (\theta, G_1,\eta_{(w,w_1)})+\textnormal {mult} (\theta, G_2,\eta_{(w,w_1)}),\notag\\
\notag\\
\textnormal {mult} (\theta, G\setminus u,\eta_{(w,w_1)})&=\textnormal {mult} (\theta, G_1\setminus u,\eta_{(w,w_1)})+\textnormal {mult} (\theta, G_2,\eta_{(w,w_1)}),\notag\\
\notag\\
\textnormal {mult} (\theta, G\setminus v,\eta_{(w,w_1)})&=\textnormal {mult} (\theta, G_1,\eta_{(w,w_1)})+\textnormal {mult} (\theta, G_2\setminus v,\eta_{(w,w_1)}),\notag\\
\notag\\
\textnormal {mult} (\theta, G\setminus uv,\eta_{(w,w_1)})&=\textnormal {mult} (\theta, G_1\setminus u,\eta_{(w,w_1)})+\textnormal {mult} (\theta, G_2\setminus v,\eta_{(w,w_1)}).\notag
\end{align}
Suppose $v$ is $(\theta,w,w_1)$-essential in $G$. Then $\textnormal {mult} (\theta, G\setminus v,\eta_{(w,w_1)})=\textnormal {mult} (\theta, G,\eta_{(w,w_1)})-1$. This implies that $\textnormal {mult} (\theta, G_2\setminus v,\eta_{(w,w_1)})=\textnormal {mult} (\theta, G_2,\eta_{(w,w_1)})-1$, and thus $\textnormal {mult} (\theta, G\setminus uv,\eta_{(w,w_1)})=\textnormal {mult} (\theta, G\setminus u,\eta_{(w,w_1)})-1$. Hence $v$ is $(\theta,w,w_1)$-essential in $G\setminus u$. The converse is proved similarly.
\end{proof}

For the proof of the following lemma, we shall use similar ideas as in \cite[Theorem 1.5]{KC}, that is by using edge manipulation and assuming first that the special vertex is of degree two. However, we cannot use the same argument as in \cite{KC} directly, because Chen and Ku assumed that $\theta\neq 0$ in their proof (in our case this is equivalent to $\theta\neq w_1(u)$ where $u$ is the $(\theta, w,w_1)$-special vertex).

\begin{lm}\label{pre_Gallai_Edmond_lemma} Let $u$ be $(\theta,w,w_1)$-special in $G$ and the degree of $u$ is two. Then $v$ is $(\theta,w,w_1)$-essential in $G\setminus u$ if and only if  $v$ is $(\theta,w,w_1)$-essential in $G$.
\end{lm}

\begin{proof}  Let $\textnormal {mult} (\theta, G,\eta_{(w,w_1)})=k$. Suppose $v$ is $(\theta,w,w_1)$-essential in $G\setminus u$. Then by Lemma \ref{special_vertex_positive}  $\textnormal{mult}(\theta, G\setminus uv,\eta_{(w,w_1)})=k$. By Lemma \ref{special_deletion_lemma}, $v$ is either $(\theta,w,w_1)$-positive  or $(\theta,w,w_1)$-essential in $G$. Suppose $v$ is $(\theta,w,w_1)$-positive in $G$. We shall show that this cannot happen.

Let  $z_1,z_2$ be the two vertices adjacent to $u$. Without loss of generality, we assume $z_1$ is $(\theta,w,w_1)$-essential in $G$. First note that $\textnormal {mult} (\theta, G\setminus v,\eta_{(w,w_1)})=k+1$ and $\textnormal {mult} (\theta, G\setminus vuz_2,\eta_{(w,w_1)})\geq k$ (Corollary \ref{path_interlacing}).  Let $G'=G-e_{uz_2}$.  Now we show that $z_2$ is $(\theta,w,w_1)$-essential and the path $p=z_1uz_2$ is $(\theta,w,w_1)$-essential in $G$. Suppose $p=z_1uz_2$ is not $(\theta,w,w_1)$-essential in $G$. Then by Lemma \ref{technical_lemma_1}, $u$ is $(\theta,w,w_1)$-special $G'$, a contrary to Lemma \ref{degree_special_vertex} (for $u$ is of degree one in $G'$). Hence the path $p=z_1uz_2$ is $(\theta,w,w_1)$-essential in $G$. By part (b)  of Lemma \ref{general_properties_path}, $z_2$ is $(\theta,w,w_1)$-essential in $G$. This also means that $v\neq z_1,z_2$ (for we assume $v$ to be $(\theta,w,w_1)$-positive in $G$).

Now we show that $\textnormal {mult} (\theta, G\setminus z_1z_2,\eta_{(w,w_1)})\geq k-1$. Note that $\textnormal {mult} (\theta, G\setminus z_1,\eta_{(w,w_1)})=k-1$. So by Lemma \ref{interlacing_root}, $\textnormal {mult} (\theta, G\setminus z_1z_2,\eta_{(w,w_1)})=k-2$, $k-1$ or $k$. Suppose $\textnormal {mult} (\theta, G\setminus z_1z_2,\eta_{(w,w_1)})=k-2$. Note that $u$ is an isolated vertex in $G\setminus z_1z_2$. So  $\textnormal {mult} (\theta, G\setminus z_1z_2u,\eta_{(w,w_1)})\leq k-2$ (Lemma \ref{isolated_vertex}). But this contradicts the conclusion of the previous paragraph that $\textnormal {mult} (\theta, G\setminus z_1uz_2,\eta_{(w,w_1)})=k-1$ (the path $p=z_1uz_2$ is $(\theta,w,w_1)$-essential in $G$). Hence $\textnormal {mult} (\theta, G\setminus z_1z_2,\eta_{(w,w_1)})\geq k-1$.

Now we show that $p=z_1uz_2$ is the only path with endpoints  $z_1$ and $z_2$ in $G$. Suppose the contrary. Then there exits a path $q_1\neq p$ with endpoints $z_{1}$ and $z_{2}$. Note that $q_1$ does not contain the vertex $u$. By Lemma \ref{essential_vertex_path}, $q_1$ is $(\theta,w,w_1)$-essential in $G$, that is $\textnormal {mult} (\theta, G\setminus q_1,\eta_{(w,w_1)})=k-1$. Since $u$ is an isolated vertex in $G\setminus q$, by Lemma \ref{isolated_vertex}, $\textnormal {mult} (\theta, G\setminus uq_1,\eta_{(w,w_1)})\leq k-1$. Since $q_1$ is a path that begins with $z_1$ and ends with $z_2$, $uq_1$ is a path that begins with $u$ and ends with $z_2$. But by part (b) of Lemma \ref{general_properties_path},
$\textnormal {mult} (\theta, G\setminus uq_1,\eta_{(w,w_1)})\geq k$, a contradiction. Hence  $p=z_1uz_2$ is the only path with endpoints  $z_1$ and $z_2$ in $G$.

Recall that $G'=G-e_{uz_2}$. So $u$ is  $(\theta,w,w_1)$-positive in $G'$ and $\textnormal {mult} (\theta, G',\eta_{(w,w_1)})=k$ (Lemma \ref{technical_lemma_1}). Next $\textnormal {mult} (\theta, G'\setminus z_2,\eta_{(w,w_1)})=\textnormal {mult} (\theta, G\setminus z_2,\eta_{(w,w_1)})=k-1$ (for $z_2$ is $(\theta,w,w_1)$-essential in $G$). So $z_2$ is  $(\theta,w,w_1)$-essential in $G'$.

Recall that we assume $v$ is $(\theta, w, w_{1})$-positive in $G$. So $\textnormal{mult}(\theta, G \setminus v, \eta_{(w, w_{1})}) =k+1$ and by Corollary 4.6, $\textnormal{mult}(\theta, G \setminus vuz_{2}, \eta_{(w,w_{1})}) \ge k$. From $\eta_{(w,w_{1})}(G \setminus v,x) = \eta_{(w,w_{1})}(G' \setminus v,x) - \left| w(e_{uz_{2}})\right|^{2}\eta_{(w,w_{1})}(G \setminus vuz_{2}, x)$ (part (b) of Theorem 2.2), we deduce that $\textnormal{mult}(\theta, G' \setminus v, \eta_{w,w_{1}}) \ge k$. This means that $v$ is either $(\theta, w,w_{1})$-neutral or $(\theta, w,w_{1})$-positive in $G'$.

Now, by part (a) of Lemma 5.2 or part (a) of Lemma 5.3, $z_{2}$ is $(\theta, w, w_{1})$-essential in $G' \setminus v$. Note that $G' \setminus v$ is a union of two graphs, say $G_{1}$ and $G_{2}$, where $u, z_{1} \in G_{1}$ and $z_{2} \in G_{2}$ (for $p=z_{1}uz_{2}$ is the only path with endpoints $z_{1}$ and $z_{2}$ in $G$). By Lemma 5.9, $z_{2}$ is $(\theta, w, w_{1})$-essential in $G' \setminus vu$. So $\textnormal{mult}(\theta, G' \setminus vuz_{2}, \eta_{(w,w_{1})}) = k-1$. But this contradicts the fact that $\textnormal{mult}(\theta, G' \setminus vuz_{2}, \eta_{(w,w_{1})})=\textnormal{mult}(\theta, G \setminus vuz_{2}, \eta_{(w,w_{1})}) \ge k$ obtained in the preceding paragraph.

Hence $v$ is $(\theta,w,w_1)$-essential in $G$.

The converse of the lemma follows from Lemma \ref{special_vertex_positive} and part (a) of Lemma \ref{positive_vertex_deletion}.
\end{proof}

\begin{thm}\label{Gallai_Edmond} Let $u$ be $(\theta,w,w_1)$-special in $G$. Then $v$ is $(\theta,w,w_1)$-essential in $G\setminus u$ if and only if  $v$ is $(\theta,w,w_1)$-essential in $G$.
\end{thm}

\begin{proof} For any vertex $z$, we shall denote its degree by $\textnormal {deg} (z)$. For any graph $G_1$, let
\begin {equation}
\chi(G_1)=\sum_{\substack {z\in V(G_1),\\ \text {$z$ is $(\theta,w,w_1)$-special in $G_1$}}} \textnormal {deg} (z),\notag
\end {equation}
and $m_{G_1}$ be the number of $(\theta,w,w_1)$-special vertex in $G_1$. By Lemma \ref{degree_special_vertex}, $\chi(G_1)\geq 2m_{G_1}$.

We shall prove the theorem by induction on $\chi(G)$. If $\chi(G)=2m_G$, then by Lemma \ref{pre_Gallai_Edmond_lemma}, we are done. Assume that the theorem holds for any graph $G_1$ with $\chi(G_1)<\chi(G)$.

Let $\textnormal {mult} (\theta, G,\eta_{(w,w_1)})=k$. If $\textnormal {deg} (u)=2$ in $G$, we are done by Lemma \ref{pre_Gallai_Edmond_lemma}. So we may assume $\textnormal {deg} (u)\geq 3$.
Let $z_1$ be a $(\theta,w,w_1)$-essential vertex adjacent to $u$. Suppose there is a vertex $z_3$ adjacent to $u$ for which the path $p=z_1uz_3$ is a not $(\theta,w,w_1)$-essential  in $G$. Then by  Lemma \ref{technical_lemma_1},  $u$ is $(\theta,w,w_1)$-special  in $G'$ and $\textnormal {mult} (\theta, G',\eta_{(w,w_1)})=k$ where $G'=G-e_{uz_{3}}$.

Suppose for all vertices $z'$ adjacent to $u$, the path $p'=z_1uz'$ is  $(\theta,w,w_1)$-essential  in $G$. Let $z_2$ and $z_4$ be adjacent to $u$. By Lemma \ref{technical_lemma_1}, $u$ is $(\theta,w,w_1)$-positive  in $G''$ and $\textnormal {mult} (\theta, G'',\eta_{(w,w_1)})=k$, where $G''=G-e_{uz_4}$. Now $\textnormal {mult} (\theta, G''\setminus z_1uz_2,\eta_{(w,w_1)})=k-1=\textnormal {mult} (\theta, G\setminus z_1uz_2,\eta_{(w,w_1)})$ (for the path $q=z_1uz_2$ is  $(\theta,w,w_1)$-essential  in $G$). So $q=z_1uz_2$ is  $(\theta,w,w_1)$-essential  in $G''$ and by part (b) of Lemma \ref{general_properties_path}, $z_1$ and $z_2$ are $(\theta,w,w_1)$-essential  in $G''$. This implies that $u$ is $(\theta,w,w_1)$-special  in $G''$.

Note that in either cases there is a vertex $z$ adjacent to $u$ such that $u$ is $(\theta,w,w_1)$-special in $G'''=G-e_{uz}$, $\textnormal {mult} (\theta, G''',\eta_{(w,w_1)})=k$ and $\chi(G''')<\chi(G)$.

Now let $v$ be $(\theta,w,w_1)$-essential in $G\setminus u$. Then $v$ is $(\theta,w,w_1)$-essential in $G'''\setminus u=G\setminus u$.  By induction hypothesis, $v$ is $(\theta,w,w_1)$-essential in $G'''$. Therefore $\textnormal {mult} (\theta, G'''\setminus v,\eta_{(w,w_1)})=k-1$.
By Lemma \ref{special_deletion_lemma}, $v$ is either $(\theta,w,w_1)$-positive  or $(\theta,w,w_1)$-essential in $G$. Suppose $v$ is $(\theta,w,w_1)$-positive in $G$.  Then by Corollary \ref{path_interlacing}, $\textnormal {mult} (\theta, G\setminus vuz,\eta_{(w,w_1)})\geq k$. But from $\eta_{(w,w_1)} ( G\setminus v,x)=\eta_{(w,w_1)} (  G'''\setminus v,x)-\vert w(e_{uz})\vert^2\eta_{(w,w_1)} (  G\setminus vuz, x)$ (part (b) of Theorem \ref{basic_recurrence}), we deduce  that $\textnormal {mult} (\theta, G\setminus v,\eta_{(w,w_1)})=k-1$, a contradiction. Hence $v$ is $(\theta,w,w_1)$-essential in $G$.

The converse of the theorem follows from Lemma \ref{special_vertex_positive} and part (a) of Lemma \ref{positive_vertex_deletion}.
\end{proof}

The following Corollary follows from Theorem \ref{Gallai_Edmond} and Lemma \ref{positive_vertex_deletion}.

\begin{cor}\label{Gallai_Edmond_decomposition}\textnormal{(Stability Lemma)} Let $G$ be a graph. If $u\in A_{(\theta,w,w_1)} (G)$ then
\begin {itemize}
\item [(i)] $D_{(\theta,w,w_1)}(G\setminus u)=D_{(\theta,w,w_1)}(G)$,
\item [(ii)] $P_{(\theta,w,w_1)}(G\setminus u)=P_{(\theta,w,w_1)}(G)$,
\item [(iii)] $N_{(\theta,w,w_1)}(G\setminus u)=N_{(\theta,w,w_1)}(G)$,
\item [(iv)] $A_{(\theta,w,w_1)}(G\setminus u)=A_{(\theta,w,w_1)}(G)\setminus \{u\}$.\qed
\end {itemize}
\end{cor}

The next corollary is a generalization of \cite[Theorem 1.7]{KC} and it can be proved using similar argument as in \cite{KC}. Nevertheless we shall give the details of the proof.

\begin{cor}\label{Gallai_critical}\textnormal{(Gallai's Lemma)} Let $G$ be a connected graph for which all vertices are $(\theta,w,w_1)$-essential. Then $G$ is $(\theta,w,w_1)$-critical.
\end{cor}

\begin{proof} Suppose $G$ is not $(\theta,w,w_1)$-critical. Then $\textnormal {mult} (\theta, G,\eta_{(w,w_1)})=k\geq 2$ (Definition \ref{vertex_classification_dfn}). Now let $v\in V(G)$. Then $\textnormal {mult} (\theta, G\setminus v,\eta_{(w,w_1)})=k-1\geq 1$. Since $G$ is connected, $v$ is not an isolated vertex. Let $u$ be a vertex adjacent to $v$ in $G$. By Corollary \ref{path_interlacing},
$\textnormal {mult} (\theta, G\setminus uv,\eta_{(w,w_1)})\geq k-1$. This implies that $u$ is either $(\theta,w,w_1)$-neutral or $(\theta,w,w_1)$-positive in $G\setminus v$. This also means that all the vertices that are adjacent to $v$ must be either $(\theta,w,w_1)$-neutral or $(\theta,w,w_1)$-positive in $G\setminus v$.

Since $\textnormal {mult} (\theta, G\setminus v,\eta_{(w,w_1)})\geq 1$, by Lemma \ref{essential_vertex_exists}, $G\setminus v$ has at least one $(\theta,w,w_1)$-essential vertex. Together with the conclusion of the previous paragraph, we deduce that $G\setminus v$ has at least one $(\theta,w,w_1)$-special vertex.

Let $A=A_{(\theta,w,w_1)}(G\setminus v)$. By Corollary \ref{Gallai_Edmond_decomposition}, a $(\theta,w,w_1)$-essential vertex remains $(\theta,w,w_1)$-essential, a $(\theta,w,w_1)$-positive vertex remains $(\theta,w,w_1)$-positive and a $(\theta,w,w_1)$-neutral vertex remains $(\theta,w,w_1)$-neutral, upon deletion of a $(\theta,w,w_1)$-special vertex. Also  a $(\theta,w,w_1)$-special vertex remains $(\theta,w,w_1)$-special, upon deletion of a $(\theta,w,w_1)$-special vertex. Therefore  if $H$ is a component in  $(G\setminus v)\setminus A$, either  $\textnormal {mult} (\theta, H)>0$ and all the vertices in $H$ are $(\theta,w,w_1)$-essential, or $\textnormal {mult} (\theta, H)=0$.

Let $Q_1, Q_2,\dots, Q_l, T_1, T_2, \dots, T_m$ be all the components in $(G\setminus v)\setminus A$ where $\textnormal {mult} (\theta, Q_j)>0$ and $\textnormal {mult} (\theta, T_{j'})=0$ for all $j,j'$.
By part (a) of Theorem \ref{basic_recurrence}, we deduce that
\begin {equation}
\textnormal {mult} (\theta, (G\setminus v)\setminus A,\eta_{(w,w_1)})=\sum_{j=1}^l \textnormal {mult} (\theta, Q_j,\eta_{(w,w_1)}).\notag
\end {equation}
On the other hand, by applying Corollary \ref{Gallai_Edmond_decomposition} repeatedly (also  Lemma \ref{special_vertex_positive}), $\textnormal {mult} (\theta, (G\setminus v)\setminus A)=k-1+\vert A\vert$. So $\sum_{j=1}^l \textnormal {mult} (\theta, Q_j)=k-1+\vert A\vert$.

Let $a\in A$. Note that $v$ is not adjacent to any vertices in $\bigcup_i Q_i$ (by the conclusion of the first paragraph). Therefore all the $Q_i$'s are components in $G\setminus A$. Since $\textnormal {mult} (\theta, G\setminus a,\eta_{(w,w_1)})=k-1$, by applying Lemma \ref{interlacing_root} repeatedly,
\begin {equation}
\textnormal {mult} (\theta, (G\setminus a)\setminus (A\setminus \{a\}), \eta_{(w, w_{1})})\leq k-1+\vert A\setminus \{a\}\vert=k-2+\vert A\vert .\notag
\end {equation}
Again by part (a) of Theorem \ref{basic_recurrence}, we deduce that $\sum_{j=1}^l \textnormal {mult} (\theta, Q_j,\eta_{(w,w_1)})\leq k-2+\vert A\vert$, a contradiction. Hence $k=1$ and $G$ is $(\theta,w,w_1)$-critical.
\end{proof}

As a consequence of Corollary \ref{Gallai_Edmond_decomposition} and Corollary \ref{Gallai_critical}, we have the following;
\begin{cor}\label {P:C7} Let $A=A_{(\theta,w,w_1)}(G)$. Then
\begin {itemize}
\item [\textnormal{(a)}]  $A_{(\theta,w,w_1)}(G\setminus A)=\varnothing$, $D_{(\theta,w,w_1)}(G\setminus A)=D_{(\theta,w,w_1)}(G)$, $P_{(\theta,w,w_1)}(G\setminus A)=P_{(\theta,w,w_1)}(G)$, and $N_{(\theta,w,w_1)}(G\setminus A)=N_{(\theta,w,w_1)}(G)$.
\item [\textnormal{(b)}] $G\setminus A$ has exactly $\left(\vert A\vert+\textnormal {mult} (\theta, G,\eta_{(w,w_1)})\right)$ $(\theta,w,w_1)$-critical components.
\item [\textnormal{(c)}] If $H$ is a component of $G\setminus A$ then either $H$ is $(\theta,w,w_1)$-critical or $\textnormal {mult} (\theta, H,\eta_{(w,w_1)})=0$.
\item [\textnormal{(d)}] The subgraph induced by $D_{(\theta,w,w_1)}(G)$ consists of all the  $(\theta,w,w_1)$-critical components in $G\setminus A$.
\end {itemize}\qed
\end{cor}

\section{Connection with classical Gallai-Edmonds decomposition}

\begin{dfn}\label{dfn_deficiency} The \emph{deficiency} of a graph, denoted by $\textnormal{def} (G)$ is the number of vertices left uncovered by any maximum matching in $G$.
\end{dfn}

The following lemma follows easily from Lemma \ref{different_form_generalized_matching}.

\begin{lm}\label{zero_root_deficiency} For any edge weight function $w$, $\textnormal{mult}(0,G,\mu_w)=\textnormal{def} (G)$.\qed
\end{lm}

By Lemma \ref{zero_root_deficiency}, the edge weight function has no effect on the multiplicity of 0 as a root in $\mu_w(G,x)$. Assume that $w_1(u)=0$ for all $u\in V(G)$. Then by Lemma \ref{weighted_to_edge_weighted_matching},
$\eta_{(w,w_1)}(G,x) =\mu_{w} (G,x)$. Note also that $D_{(0,w,w_1)}(G)$ is the set of all vertices in $G$ which are not covered by at least one maximum matching of $G$. Also $N_{(0,w,w_1)}(G)=\varnothing$, for otherwise, there would be a vertex say $u$ with $\textnormal{mult} (0,G,\mu_w)=\textnormal{mult} (0,G\setminus u,\mu_w)$. But this means that there is a maximum matching that does not cover $u$ and so $u\in D_{(0,w,w_1)}(G)$, a contradiction (see \cite[Section 3.2 on p. 93]{Lo} for the details). Therefore $V(G)=D_{(0,w,w_1)}(G)\cup A_{(0,w,w_1)}(G)\cup P_{(0,w,w_1)}(G)$ which is the classical Gallai-Edmonds decomposition provided $w(e)=1$ for all $e \in E(G)$. Also  \cite[Lemma 3.2.2 on p. 95]{Lo} is a special case of the Stability Lemma (Corollary \ref{Gallai_Edmond_decomposition}).

\begin{lm}\label{fixed_weight_on_vertex} Suppose $w_1(u)=c$ for all $u\in V(G)$, where $c$ is a constant real number. Then $\eta_{(w,w_1)}(G,x+c) =\mu_{w} (G,x)$.
\end{lm}

\begin{proof} We shall prove by induction on $\vert V(G)\vert$. Suppose $\vert V(G)\vert=1$. Then $\eta_{(w,w_1)}(G,x)=x-c$. Therefore $\eta_{(w,w_1)}(G,x+c) =x=\mu_{w} (G,x)$. Assume it is true for all graphs with fewer vertices than $G$.

Let $u,v\in V(G)$. By induction, $\eta_{(w,w_1)}(G \setminus u,x+c)=\mu_w(G \setminus u,x)$ and $\eta_{(w,w_1)}(G \setminus uv,x+c)=\mu_w(G \setminus uv,x)$. So  $\eta_{(w,w_1)}(G,x+c) = x \mu_w(G \setminus u,x) - \sum_{v \sim u} \vert w(e_{uv})\vert^2 \mu_w(G \setminus uv,x)$ (by part (c) of Theorem \ref{basic_recurrence}). It then follows from by part (c) of Lemma \ref{pre_basic_recurrence}, that $\eta_{(w,w_1)}(G,x+c) =\mu_{w} (G,x)$.
\end{proof}

The next lemma follows from Lemma \ref{zero_root_deficiency} and Lemma \ref{fixed_weight_on_vertex}.

\begin{lm}\label{fixed_weight_vertex_root_deficiency} Suppose $w_1(u)=c$ for all $u\in V(G)$, where $c$ is a constant real number.  Then for any edge weight function $w$, $\textnormal{mult}(c,G,\eta_{(w,w_1)})=\textnormal{def} (G)$.\qed
\end{lm}

As a consequence of Lemma \ref{fixed_weight_vertex_root_deficiency}, we see that if the weight on each vertex is a  constant, say $c$, then the edge weight function has no effect on the multiplicity of $c$ as a root of $\eta_{(w,w_1)}(G,x)$. In fact, it depends only on the structure of the graph. Note also that $D_{(c,w,w_1)}(G)$ is the set of all vertices in $G$ which are not covered by at least one maximum matching of $G$, $N_{(c,w,w_1)}(G)=\varnothing$  and  $V(G)=D_{(c,w,w_1)}(G)\cup A_{(c,w,w_1)}(G)\cup P_{(c,w,w_1)}(G)$.

\section{Connection with the Parter-Wiener theorem}

In separate papers, Parter \cite{P} and Wiener \cite{W} independently observed an important theorem about the existence of principal submatrices of a Hermitian matrix whose graph is a tree, in which the multiplicity of an eigenvalue increases. Recently, Johnson, Duarte and Saiago \cite{Johnson2} generalized this result by providing more structural information. It turns out that these results are just special cases of the Gallai-Edmonds structure theorem that we have developed in this paper.

Given an $n\times n$ Hermitian matrix $B=[b_{uv}]$, we can associate a graph $G$ to it as follows: let $G$ be the graph with $V(G)=\{1,2,\dots, n\}$ and $e_{uv}\in E(G)$ if and only if $b_{uv}\neq 0$, $u \not = v$. Clearly, for a given graph $G$, there are many Hermitian matrix $B=[b_{uv}]$ whose associated graph is $G$; moreover, we shall assign weights to $G$ using $B$ as follows: set $w(e_{uv})=b_{uv}$ if $e_{uv}\in E(G)$ with $u<v$ and set $w_1(u)=b_{uu}$ for all $u\in V(G)$. Consequently, $B=B_{(w,w_1)}(G)$ is the weighted adjacency matrix of $G$ so that the characteristic polynomial of $B$ is $\phi_{(w,w_1)}(G,x)$ and the eigenvalues of $B$ are the roots of $\phi_{(w,w_1)}(G,x)$.

For the rest of this section, let $T$ be a tree on $n$ vertices $1, 2, \ldots, n$ and suppose that $\mathcal{S}(T)$ is the set of all $n \times n$ Hermitian matrix whose graph is $T$. Let $m_{B}(\lambda)$ denote the multiplicity of $\lambda$ as an eigenvalue of $B$. Suppose $B \in \mathcal{S}(T)$. Let $B(u;u)$ be the matrix obtained from $B$ by deleting the $u$-th row and $u$-column. Note that $B(u;u)=B_{(w,w_1)}(T\setminus u)$.

In the literature, a vertex $v$ of a tree $T$ is called \emph{parter} if $m_{B(v;v)}(\lambda)=m_B(\lambda)+1$, \emph{neutral} if $m_{B(v;v)}(\lambda)=m_B(\lambda)$ and \emph{downer} if $m_{B(v;v)}(\lambda)=m_B(\lambda)-1$ (see \cite{Johnson, Johnson2, Johnson3, Johnson4}). On the other hand, by Corollary \ref{tree_matching_characteristic_identical}, $\phi_{(w,w_1)}(T,x)=\eta_{(w,w_1)}(T,x)$.  So parter, neutral and downer (relative to the eigenvalue $\lambda$) are just $(\lambda,w,w_1)$-positive, $(\lambda,w,w_1)$-neutral and $(\lambda,w,w_1)$-essential respectively, in the language of Godsil (Definition \ref{vertex_classification_dfn}). Furthermore, $m_{B}(\lambda) = \textnormal{mult}(\lambda, T, \eta_{(w, w_{1})})$.

The following result has been important in the recent development on possible multiplicities of eigenvalues among matrices in $\mathcal{S}(T)$ \cite{Johnson, Johnson2, Johnson3, Johnson4}.

\begin{thm}[Parter-Wiener]
Let $T$ be a tree on $n$ vertices and suppose $B \in \mathcal{S}(T)$ and $\lambda \in \mathbb{R}$ is such that $m_{B}(\lambda) \ge 2$. Then, there is a vertex $u$ of $T$ such that $m_{B(u;u)}(\lambda) = m_{B}(\lambda)+1$ and $\lambda$ occurs as an eigenvalue in direct summands of $B$ that correspond to at least three branches of $T$ at $u$.\hfill\qed
\end{thm}

A more general statement was proved by Johnson, Duarte and Saiago \cite{Johnson2} as follows.

\begin{thm}[Johnson-Duarte-Saiago]
Let $B$ be a Hermitian matrix whose graph is $T$, and suppose that there exists a vertex $u$ of $T$ and a real number $\lambda$ such that $\lambda \in \sigma(B) \cap \sigma(B(u;u))$ (here $\sigma(B)$ denotes the set of all eigenvalues of $B$). Then
\begin{itemize}
\item[\textnormal{(a)}] there is a vertex $v$ of $T$ such that $m_{B(v;v)}(\lambda) = m_{B}(\lambda)+1$;
\item[\textnormal{(b)}] if $m_{B}(\lambda) \ge 2$, then $v$ may be chosen so that the degree of $v$ is at least $3$ and so that there are at least three components $T_{1}$, $T_{2}$ and $T_{3}$ of $T-v$ such that $m_{B[T_{i}]}(\lambda) \ge 1$, $i=1,2,3$ (here $B[T_{i}]$ is the principal submatrix of $B$ from retention of the rows and columns which correspond to $T_{i}$ );
\item[\textnormal{(c)}] if $m_{B}(\lambda)=1$, then $v$ may be chosen so that the degree of $v$ is at least $2$ and so that there are two components $T_{1}$ and $T_{2}$ of $T-v$ such that $m_{B[T_{i}]}(\lambda) = 1$, $i=1,2$.
\end{itemize}
\end{thm}

\begin{proof}
We sketch a proof based on the Gallai-Edmonds structure theorem (Corollary \ref{Gallai_Edmond_decomposition}, Corollary \ref{Gallai_critical}, Corollary \ref{P:C7}). Note that the condition $\lambda \in \sigma(B) \cap \sigma(B(u;u))$ means that either $m_{B}(\lambda) \ge 2$ or $m_{B}(\lambda)=1$ and $u$ is not $(\lambda, w, w_{1})$-essential in $T$. In the later, we observe that $A_{(\lambda, w, w_{1})}(T) \not = \emptyset$ since $T$ is connected. In fact, this is equivalent to saying that $\lambda$ is a root of $\eta_{(w,w_{1})}(T,x)$ and $A_{(\lambda, w,w_{1})}(T) \not = \emptyset$ (see Corollary \ref{Gallai_critical}). We shall consider these cases separately.

\noindent {\em Case I.}~~$m_{B}(\lambda) \ge 2$.

If $ A_{(\lambda, w, w_{1})}(T) = \{v\}$ then it follows immediately from part (b) of Corollary \ref{P:C7} that $v$ has the property required by (b). We proceed to prove the theorem by induction on $\left|A_{(\lambda, w, w_{1})}(T) \right|$.

Fix $v \in A_{(\lambda, w, w_{1})}(T)$ and consider a component $T'$ of $T \setminus v$ with $A_{(\lambda, w, w_{1})}(T') \not = \emptyset$. Such $T'$ exists since $A_{(\lambda, w, w_{1})}(T\setminus v) = A_{(\lambda,w,w_{1})}(T) \setminus v \not = \emptyset$ (Corollary \ref{Gallai_Edmond_decomposition}). By the inductive hypothesis, there exists $v' \in A_{(\lambda, w,w_{1})}(T')$ such that $T' \setminus v'$ consists of $T'_{1}$, $\ldots$, $T'_{k}$, $S'_{1}$, $\ldots$, $S'_{l}$ with $m_{B[T'_{i}]}(\lambda) \ge 1$ for all $1 \le i \le k$, $k \ge 3$, and $m_{B[S'_{i}]}(\lambda) = 0$ for all $1 \le i \le l$ (it is possible that there does not exist any $S'_{i}$).

Since $T$ is a tree, $v$ is joined to at most one of $T'_{i}$'s, $S'_j$'s or $v'$. If $v$ is not joined to any of the $T'_{i}$'s then we are done; otherwise, we may assume that $k=3$ and $v$ is joined to $T'_{3}$. Set $T_{1}=T'_{1}$, $T_{2}=T'_{2}$ and $T_{3}$ to be the component of $T \setminus v'$ containing $T'_{3}$ and $v$. It is readily deduced from $m_{B[T \setminus v']}(\lambda) \ge 3$ that $m_{B[T_{i}]}(\lambda) \ge 1$ for $i=1,2,3$, as desired.

\noindent {\em Case II.}~~$m_{B}(\lambda)=1$.

Since a similar argument can be used to settle this case, we omit the details.
\end{proof}

\end{document}